\documentclass[12pt,leqno]{amsart}
\usepackage[a4paper,margin=1in]{geometry}
\usepackage{amsmath,amssymb,amsthm,mathtools}
\usepackage{bm}
\usepackage{graphicx}
\usepackage{float}
\usepackage{hyperref}
\usepackage{enumitem}
\usepackage{booktabs}
\usepackage{cite}
\hypersetup{colorlinks=true,linkcolor=blue,citecolor=blue,urlcolor=blue}

\newtheorem{theorem}{Theorem}[section]
\newtheorem{proposition}[theorem]{Proposition}
\newtheorem{lemma}[theorem]{Lemma}
\newtheorem{corollary}[theorem]{Corollary}
\theoremstyle{definition}

\newtheorem{remark}[theorem]{Remark}

\newcommand{\cT}{\mathcal T}
\numberwithin{equation}{section}

\title[Mixed Torsion Cells and Polygonal Monotonicity]{Mixed Torsion on Right Triangles and the P\'olya--Szeg\H{o} Monotonicity Problem for Regular Polygons}
\author{Changfeng Gui}
\address{Department  of  Mathematics,  University  of  Macau,  Macau  SAR,  P. R. China}
\address{Zhuhai UM Science and Technology Research Institute, Hengqin, Guangdong, 519031,  P. R. China}
\email{changfenggui@um.edu.mo}
\author{Yeyao Hu}
\address{School of Mathematics and Statistics, HNP-LAMA, Central South University, Changsha, Hunan 410083, P. R. China}
\email{huyeyao@gmail.com}
\author{Qinfeng Li}
\address{School of Mathematics, Hunan University, Changsha, P.R. China.}
\email{liqinfeng1989@gmail.com}
\author{Chenyang Zhang}
\address{School of Mathematics, Hunan University, Changsha, P.R. China.}
\email{2020531070@qq.com}
\date{}

\begin{document}

	\begin{abstract}
		Motivated by the polygonal P\'olya--Szeg\H{o} conjecture for torsional rigidity, we study two monotonicity problems for torsional rigidity.  The first concerns a mixed torsion problem on fixed-area right triangles, with a Dirichlet condition on one leg and Neumann conditions on the other leg and on	the hypotenuse.  We prove that the mixed torsional rigidity strictly increases as the ratio of the Neumann leg to the Dirichlet leg increases.  The proof	uses a Hadamard shape derivative, a Pohozaev-type identity, and a monotonicity result for the mixed torsion function. We also prove a similar result for the mixed ground state of Laplacian.
		
		The second concerns regular polygons.  If \(P_N\) denotes the regular
		\(N\)-gon of area \(\pi\), we prove, by a purely analytic
		Schwarz--Christoffel/Bergman analytic-content argument, that
		\[
		T^D(P_{N+1})>T^D(P_N),\qquad N\ge3,
		\]
		where \(T^D\) is the Dirichlet torsional rigidity.  We also obtain the
		asymptotic expansion
		\[
		T^D(P_N)=\frac{\pi}{8}-\frac{\pi\zeta(3)}{N^3}
		+\frac{\pi^5}{45N^4}+O(N^{-5}).
		\]
	\end{abstract}

	\maketitle
	
	\section{Introduction}
	
	The starting point of this paper is the polygonal P\'olya--Szeg\H{o} conjecture for torsional rigidity.  If \(\Omega\subset\mathbb R^2\) is a bounded Lipschitz domain, let
	\[
	T^D(\Omega)=\int_\Omega u\,dx,
	\]
	where \(u\) is the Dirichlet torsion function,
	\[
	-\Delta u=1\quad\hbox{in }\Omega,
	\qquad
	u=0\quad\hbox{on }\partial\Omega .
	\]  The conjecture asserts that, among all \(N\)-gons of a prescribed area, the regular \(N\)-gon uniquely maximizes \(T^D\).  This is the torsional counterpart of the polygonal Faber--Krahn problem for the first Dirichlet eigenvalue.  P\'olya and Szeg\H{o} proved the classical isoperimetric inequalities for torsion and principal frequency and settled the polygonal extremal problems in the triangular and quadrilateral cases by symmetrization methods \cite{Polya1948,PolyaSzego}.  For \(N\ge5\), however, the fixed-\(N\) polygonal torsion problem remains open in general; among the closely related P\'olya--Szeg\H{o} polygonal conjectures, the logarithmic-capacity case was settled by Solynin--Zalgaller \cite{SolyninZalgaller}.
	
	For the polygonal Dirichlet eigenvalue problem, the direct method together with a perturbation argument gives existence of minimizers and monotonicity of the optimal values with respect to the number of sides; see Henrot \cite[Chapter~3]{HenrotBook}.  The similar argument applies to the maximization of Dirichlet torsional rigidity among fixed-area polygons.  Thus, if \(t_N\) denotes the maximal torsional rigidity among fixed-area \(N\)-gons, then \(t_N\) is attained and the sequence \(t_N\) is strictly increasing in \(N\).\footnote{The
		torsional version of this direct-method argument was pointed out to us by
		A.~Henrot.}
	The difficult open question is whether
	\[
	t_N=T^D(P_N),
	\]
	where \(P_N\) is the regular \(N\)-gon of the same area.  On the torsion side, Fragal\`a--Gazzola--Lamboley proved sharp estimates for
	\(p\)-torsion and obtained a partial polygonal P\'olya--Szeg\H{o} result for
	polygons with sufficiently large asymmetry, while Bucur--Fragal\`a established
	regular-polygon optimality for several related constrained variational
	energies on convex polygons \cite{FragalaGazzolaLamboley,BucurFragala2021}.
	Other relevant approaches include torsional inequalities for tangential
	polygons \cite{Keady} and the Bergman analytic-content viewpoint of
	Fleeman--Simanek and Kraus--Simanek, which is particularly useful for
	polygonal domains and is also used in the present paper
	\cite{FleemanSimanek,KrausSimanek}.
	
	The present paper does not solve the fixed-\(N\) maximization problem.  Instead, it addresses a closely related monotonicity question for the regular polygons themselves.  We normalize \(P_N\) to have area \(\pi\).  Since the disk of area \(\pi\) has torsional rigidity \(\pi/8\) and is the Saint-Venant maximizer among all planar domains, it is natural to ask whether
	\[
	T^D(P_3)<T^D(P_4)<T^D(P_5)<\cdots<\frac{\pi}{8}.
	\]
	Before this work, even this monotonicity of the regular-polygon sequence \(T^D(P_N)\) was not known.  This question is parallel to several recent regular-polygon monotonicity results for other spectral quantities.  Dahne--G\'omez-Serrano--Pech-Alberich proved that the first Dirichlet eigenvalue of the fixed-area regular \(N\)-gon is strictly decreasing in \(N\), settling a conjecture of Antunes--Freitas \cite{AntunesFreitas,DahneGomezSerranoPechAlberich}.  Cheng--Gui--Hu--Li--Yao proved that the first nonzero Steklov eigenvalue of the perimeter-normalized regular \(N\)-gon is strictly increasing in \(N\) \cite{ChengGuiHuLiYaoSteklov}.  Other recent work on the polygonal Faber--Krahn problem and on regular-polygon eigenvalue asymptotics includes \cite{BogoselBucur,Indrei,Nitsch,Berghaus}.  These results provide a natural background for the torsional monotonicity problem considered here.
	
	Our first observation is that the regular-polygon torsion problem is naturally linked to a mixed boundary problem on a right triangle.  For \(q>0\), let
	\[
	\Omega_q=\{(x,y):0<x<q,\ 0<y<1-x/q\},
	\]
	with vertices
	\[
	A=(0,0),\qquad B_q=(q,0),\qquad C=(0,1).
	\]
	On \(\Omega_q\) we impose a Dirichlet condition on the vertical side \(AC\) and homogeneous Neumann conditions on the horizontal side \(AB_q\) and the hypotenuse \(B_qC\):
	\begin{equation}\label{eq:intro-mixed-problem}
		\begin{cases}
			-\Delta u_q=1 & \text{in }\Omega_q,\\
			u_q=0 & \text{on }AC,\\
			\partial_\nu u_q=0 & \text{on }AB_q\cup B_qC.
		\end{cases}
	\end{equation}
	Here \(\nu\) denotes the outward unit normal on the Neumann part.  We define the mixed torsional rigidity
	\[
	T(\Omega_q)=\int_{\Omega_q}u_q\,dx\,dy,
	\qquad
	F(q):=\frac{T(\Omega_q)}{q^2}.
	\]
	Since \(|\Omega_q|=q/2\) and torsional rigidity scales like length to the fourth power, \(F(q)\) is, up to an absolute constant, the fixed-area normalization of the mixed torsional rigidity.
	
	The connection with regular polygons is exact.  Put
	\[
	\alpha=\frac{\pi}{N},\qquad q=\cot\alpha.
	\]
	The regular \(N\)-gon can be decomposed into reflected copies of \(\Omega_q\), and after scaling to area \(\pi\) one obtains the cell relation
	\begin{equation}\label{eq:intro-cell-relation}
		T^D(P_N)
		=
		2\pi\alpha\,F(\cot\alpha)
		=
		\frac{2\pi^2}{N}\,
		\frac{T(\Omega_{\cot(\pi/N)})}{\cot^2(\pi/N)}.
	\end{equation}
	Thus the discrete monotonicity of \(T^D(P_N)\) leads naturally to a continuous mixed-cell problem: is \(F(q)\) increasing for all \(q>0\)?
	
	This mixed-cell problem also has a direct physical interpretation.  The function \(u_q\) is the steady-state temperature generated by a uniform heat source in a right triangular plate when one side is kept at zero temperature and the other two sides are insulated.  After rescaling the triangle to fixed area, \(F(q)\) is proportional to the total, and hence also the average, steady-state temperature.  Therefore the question asks whether, under fixed area, increasing the ratio between the Neumann right side and the Dirichlet right side always increases the average temperature, provided that the hypotenuse is also insulated.
	
	The present mixed-cell monotonicity should be compared with the
	Dirichlet-reflected configuration in which one right side carries a Neumann condition while the other right side and the hypotenuse carry Dirichlet
	conditions.  After reflection across the Neumann side, this becomes the
	Dirichlet torsion problem on an isosceles triangle.  Using an argument based
	on continuous Steiner symmetrization, in the spirit of Siudeja's work on
	eigenvalue inequalities for triangles \cite{SiudejaTriangles}, one obtains that the fixed-area normalized torsional rigidity is
	increasing for \(0<q<\sqrt3\) and decreasing for \(q>\sqrt3\).  More recently, Huang--Li--Xie--Yang gave an independent flow and shape-derivative proof of the same monotonicity result \cite{HuangLiXieYang2025}.  Their argument
	reduces the sign of the shape derivative to a comparison of boundary gradients on the reflected Dirichlet domain.  In the present problem the hypotenuse is Neumann, and hence the same Dirichlet reflection mechanism is not available.  In fact, the answer is completely different for the mixed-cell
	problem: the fixed-area normalized mixed torsional rigidity is strictly increasing for all \(q>0\).

	\begin{theorem}[Monotonicity of the normalized mixed cell]\label{thm:intro-mixed-normalized}
		The function
		\[
		q\longmapsto F(q):=\frac{T(\Omega_q)}{q^2}
		\]
		is strictly increasing on \((0,\infty)\).  More precisely, if
		\[
		A(q)=\int_{\Omega_q}(u_q)_x^2,\qquad
		B(q)=\int_{\Omega_q}(u_q)_y^2,
		\]
		then
		\[
		\frac{d}{dq}\left(\frac{T(\Omega_q)}{q^2}\right)
		=
		\frac{A(q)-B(q)}{q^3}>0.
		\]
	\end{theorem}
	
	The proof of Theorem~\ref{thm:intro-mixed-normalized} is based on a Hadamard shape derivative and a Pohozaev-type identity.  The derivative formula gives
	\[
	\frac{d}{dq}\left(\frac{T(\Omega_q)}{q^2}\right)=\frac{A(q)-B(q)}{q^3},
	\]
	and the sign \(A(q)>B(q)\) follows from another Pohozaev identity together with the following directional sign property of the mixed torsion function.  This result is a technical input for the shape-functional monotonicity; it is proved directly by applying the maximum principle to the derivatives of \(u_q\).
	
	\begin{proposition}[Directional sign of the mixed-cell torsion function]\label{thm:intro-mixed-gradient}
		For every \(q>0\), the solution of \eqref{eq:intro-mixed-problem} satisfies
		\[
		(u_q)_x>0,
		\qquad
		(u_q)_y<0
		\qquad\text{in }\Omega_q.
		\]
		In particular, the maximum of \(u_q\) on \(\overline{\Omega_q}\) is attained at the Neumann vertex \((q,0)\).
	\end{proposition}
	
	The \(x\)-monotonicity is related to the semilinear mixed-boundary monotonicity theorem of Li--Yao for triangles with a non-obtuse Neumann vertex \cite{LiYaoMixedTriangles}.  Here the proof is torsion-specific, and gives, in addition, the vertical sign \((u_q)_y<0\).  This vertical sign is the key point for the torsional rigidity: it implies \(A(q)>B(q)\), and therefore proves the strict increase of \(F(q)\).

	We next study the large-\(q\) behavior of the mixed cell.  By translating the vertex \(B_q\) to the origin, scaling, and reflecting across the horizontal Neumann side, \(\Omega_q\) is transformed into a thin sector-like mixed domain.  A sector perturbation argument gives the following expansion.
	
	\begin{theorem}[Large-\(q\) asymptotics of the mixed cell]\label{thm:intro-triangle-asymptotics}
		As \(q\to\infty\),
		\[
		T(\Omega_q)
		=
		\frac{q^3}{16}
		+
		\frac q{48}
		-
		\frac{\zeta(3)}{2\pi^3}
		+
		O(q^{-1}).
		\]
	\end{theorem}
	
	Substituting \(q=\cot(\pi/N)\) in the cell relation \eqref{eq:intro-cell-relation} gives the leading regular-polygon asymptotic
	\[
	T^D(P_N)=\frac{\pi}{8}-\frac{\pi\zeta(3)}{N^3}+O(N^{-4}).
	\]
	To identify the next coefficient and to prove monotonicity for every \(N\ge3\), we use a second input: the Schwarz--Christoffel representation of the regular polygon combined with the Bergman analytic-content formula for torsional rigidity.  If a Taylor series for a conformal map from the disk is available, one can compute the stress function or the torsional rigidity; see P\'olya--Szeg\H{o}, Sokolnikoff, and the modern Bergman analytic-content formulation of Fleeman--Simanek \cite{PolyaSzego,Sokolnikoff,FleemanSimanek}.  The resulting exact Schwarz--Christoffel--Bergman series yields coefficient-level information and direct derivative estimates.
	
	\begin{theorem}[Asymptotics for regular polygons]\label{thm:intro-polygon-asymptotics}
		Let \(P_N\) be the regular \(N\)-gon of area \(\pi\).  Then, as \(N\to\infty\),
		\[
		T^D(P_N)
		=
		\frac{\pi}{8}
		-
		\frac{\pi\zeta(3)}{N^3}
		+
		\frac{\pi^5}{45N^4}
		+
		O(N^{-5}).
		\]
	\end{theorem}
	
	The \(N^{-3}\) term in Theorem~\ref{thm:intro-polygon-asymptotics} comes from the mixed-cell reduction and Theorem~\ref{thm:intro-triangle-asymptotics}.  The \(N^{-4}\) coefficient comes from the exact Schwarz--Christoffel--Bergman series.  More importantly, the same series can be estimated directly as an analytic function of \(t=1/N\), giving the following monotonicity theorem.
	
	\begin{theorem}[Strict monotonicity of regular-polygon torsion]\label{thm:intro-polygon-monotonicity}
		For every integer \(N\ge3\),
		\[
		T^D(P_{N+1})>T^D(P_N).
		\]
	\end{theorem}
	
	The proof of Theorem~\ref{thm:intro-polygon-monotonicity} is purely analytic.  For \(N\ge5\), the exact Schwarz--Christoffel--Bergman series is rewritten as a function \(\mathcal T(t)\), \(t=1/N\), and we prove directly that \(\mathcal T'(t)<0\) for \(0<t\le1/5\).  The remaining cases \(N=3,4,5\) are verified by explicit elementary estimates.  Unlike the proof of \cite{DahneGomezSerranoPechAlberich} for the Dirichlet eigenvalue and the proof of \cite{ChengGuiHuLiYaoSteklov} for the Steklov eigenvalue, the argument here does not use computer-assisted verification.
	
	Beyond torsion, the same mixed-cell method gives a complementary result for the first mixed Dirichlet--Neumann eigenvalue.  Let \(\mu(q)\) be the first eigenvalue of \(-\Delta\) on \(\Omega_q\), with Dirichlet condition on \(AC\) and Neumann condition on \(AB_q\cup B_qC\).
	
	\begin{theorem}[Normalized first mixed eigenvalue monotonicity]\label{monoeigenvalue}
		For every \(q>0\),
		\[
		\frac{d}{dq}\bigl(q\mu(q)\bigr)<0.
		\]
	\end{theorem}
	
	Combining Theorem~\ref{thm:intro-mixed-normalized} and Theorem~\ref{monoeigenvalue}, we obtain the following angle monotonicity.
	
	\begin{corollary}\label{cor:isoceles-angle}
		Let \(I_\alpha\) be the isosceles triangle with vertex angle \(\alpha\) and area \(1\).  If homogeneous Neumann conditions are imposed on the two equal sides and a Dirichlet condition is imposed on the base, then the mixed torsional rigidity is strictly decreasing in \(\alpha\), while the first mixed Dirichlet--Neumann eigenvalue is strictly increasing in \(\alpha\).
	\end{corollary}
	
	Indeed, after reflection one has \(q=\cot(\alpha/2)\); fixed-area normalization gives \(T(I_\alpha)=2T(\Omega_q)/q^2\) and \(\lambda_1(I_\alpha)=q\mu(q)\).  Since \(T(\Omega_q)/q^2\) is increasing in \(q\), while \(q\mu(q)\) is decreasing in \(q\), and since \(q\) is decreasing in \(\alpha\), the corollary follows.
	
	\vskip 0.3cm
	
	Finally, the maximum-principle mechanism behind Proposition~\ref{thm:intro-mixed-gradient} also gives a more general Dirichlet--Robin monotonicity principle.  We state it here as an auxiliary result at the end of the paper, separate from the main polygonal torsion line. This result extends the result in \cite{LWY}, and the proof is a direct maximum-principle argument and does not use the continuity method of \cite{LWY}.

	\vskip 0.3cm
	
	The paper is organized as follows.  Section~\ref{sec:mixed-continuity} proves Proposition~\ref{thm:intro-mixed-gradient}.  Section~\ref{sec:shape} proves Theorem~\ref{thm:intro-mixed-normalized} by combining a Hadamard shape derivative with a Pohozaev-type identity.  Section~\ref{sec:large-q} proves Theorem~\ref{thm:intro-triangle-asymptotics}.  Section~\ref{sec:polygons} derives the polygonal cell relation and the leading \(N^{-3}\) asymptotics.  Section~\ref{sec:analytic-N5} gives the exact Schwarz--Christoffel--Bergman series, completes the proof of Theorem~\ref{thm:intro-polygon-asymptotics}, and proves Theorem~\ref{thm:intro-polygon-monotonicity}.  Section~\ref{secmonoeigenvalue} proves Theorem~\ref{monoeigenvalue}.  Section~\ref{mixedrobinsection} proves the auxiliary Dirichlet--Robin monotonicity result.

	\section{A directional sign lemma for the mixed cell}
	\label{sec:mixed-continuity}

	We prove Proposition~\ref{thm:intro-mixed-gradient}.  To simplify notation,
	write
	\[
	\Omega = \{ (x,y) : x>0,\ y>0,\ x+qy<q \},
	\]
	where $q>0$.  The side $AC$ carries the Dirichlet condition, while
	$AB$ and $BC$ carry homogeneous Neumann conditions.
	
	Specifically, the boundary conditions can be explicitly written as:
	\begin{align*}
		u &= 0 \quad \text{on } AC, \\
		u_y &= 0 \quad \text{on } AB, \\
		u_x + q u_y &= 0 \quad \text{on } BC,
	\end{align*}
	where the unnormalized outward normal vector on the hypotenuse $BC$ is
	$N=(1,q)$.  We shall use the tangent vector $T=(q,-1)$, oriented from
	$C$ to $B$.
	
	We use the standard regularity theory for mixed boundary problems on
	polygons.  Since the two Dirichlet--Neumann junctions have opening angles at
	most $\pi/2$, the first derivatives of $u$ are continuous up to each open
	side and up to the junctions in the sense needed below; alternatively, the
	maximum-principle arguments may be carried out on corner-truncated triangles
	and then passed to the limit.  Thus the boundary traces used below are
	legitimate; see, for instance, Grisvard's regularity theory for polygonal
	domains~\cite{Grisvard}.
	
	\begin{proof}[Proof of Proposition~\ref{thm:intro-mixed-gradient}]
		
		\medskip\noindent\textbf{Step 1: Positivity of $u$.}
		Since $-\Delta u = 1 > 0$ in $\Omega$ and $u=0$ on the nonempty Dirichlet boundary $AC$, combined with the homogeneous Neumann conditions on the remaining boundaries, the maximum principle implies that $u \ge 0$ in $\Omega$. If $u$ vanishes at any interior point, the strong maximum principle would force $u \equiv 0$, which directly contradicts $-\Delta u = 1$. Thus, we have
		\[
		u > 0 \quad \text{in } \Omega.
		\]
		Consequently, on the Dirichlet boundary $AC$, where the outward unit normal vector is $\nu = (-1,0)$, Hopf's lemma implies that $\partial_\nu u < 0$ on the open segment $AC^\circ$. This translates to $-u_x < 0$, which yields
		\[
		u_x > 0 \quad \text{on } AC^\circ.
		\]
		
		\medskip\noindent\textbf{Step 2: Proof that $u_y<0$.}
		Let $w = u_y$. Differentiating the governing equation $-\Delta u = 1$ with respect to $y$, we find that $w$ is harmonic:
		\[
		\Delta w = 0 \quad \text{in } \Omega.
		\]
		We now examine the boundary conditions for $w$:
		\begin{itemize}
			\item On $AB$, the Neumann condition $u_y = 0$ immediately gives $w = 0$.
			\item On $AC$, since $u(0,y) = 0$, the tangential derivative along $AC$ must vanish, which implies $w = u_y = 0$.
			\item On the hypotenuse $BC$, we have $u_x + q u_y = 0$. We choose the tangent vector along $BC$ as $T = (q, -1)$. Differentiating the boundary condition along the tangential direction $T$ yields:
			\[
			T \cdot \nabla (u_x + q u_y) = 0 \implies q u_{xx} + (q^2 - 1)u_{xy} - q u_{yy} = 0.
			\]
			Utilizing the original PDE $u_{xx} + u_{yy} = -1$, we substitute $u_{xx} = -1 - u_{yy}$ into the expression above to obtain:
			\[
			(q^2 - 1)u_{xy} - 2q u_{yy} = q.
			\]
			Rewriting this in terms of $w$ (since $w_x = u_{xy}$ and $w_y = u_{yy}$), we get:
			\[
			(q^2 - 1)w_x - 2q w_y = q \quad \text{on } BC.
			\]
			Noting that the coefficient vector can be decomposed as $(q^2-1, -2q) = -N + qT$, this boundary relation can be compactly written as:
			\begin{equation}\label{eq:w_boundary}
				(-N + qT) \cdot \nabla w = q \quad \text{on } BC.
			\end{equation}
		\end{itemize}
		
		We proceed by contradiction. Suppose $w$ attains a positive maximum in $\overline{\Omega}$. Since $w$ is harmonic, the strong maximum principle dictates that this positive maximum cannot occur in the interior. Moreover, since $w = 0$ on $AB \cup AC$ and the endpoint traces are obtained by continuity, this positive maximum must be achieved at some point $P$ on the open edge $BC^\circ$. 
		
		At this boundary maximum point $P$, the tangential derivative must vanish, i.e., $T \cdot \nabla w(P) = 0$. Simultaneously, Hopf's lemma implies that the outward normal derivative must be strictly positive, i.e., $N \cdot \nabla w(P) > 0$. However, substituting $T \cdot \nabla w(P) = 0$ into \eqref{eq:w_boundary} yields:
		\[
		-N \cdot \nabla w(P) = q \implies N \cdot \nabla w(P) = -q < 0,
		\]
		which contradicts Hopf's lemma. Therefore, $w \le 0$ in $\Omega$. By the strong maximum principle, since $w$ cannot be identically zero due to the non-homogeneous boundary condition \eqref{eq:w_boundary}, we conclude that
		\[
		w = u_y < 0 \quad \text{in } \Omega.
		\]
		Furthermore, the same Hopf's lemma argument on the boundary ensures that $u_y = w < 0$ on $BC^\circ$.
		
		\medskip\noindent\textbf{Step 3: Proof that $u_x>0$.}
		Let $z = u_x$. Differentiating the main equation with respect to $x$ shows that $z$ is also harmonic:
		\[
		\Delta z = 0 \quad \text{in } \Omega.
		\]
		Let us verify the boundary behavior of $z$:
		\begin{itemize}
			\item On $AC^\circ$, Step 1 already established that $z = u_x > 0$.
			\item On $BC^\circ$, the Neumann boundary condition $u_x + q u_y = 0$ implies $z = -q u_y$. Since we proved $u_y < 0$ on $BC^\circ$ in Step 2 and $q > 0$, it follows that $z > 0$ on $BC^\circ$.
			\item On $AB$, differentiating $u_y = 0$ with respect to $x$ gives $u_{xy} = 0$, which means $z_y = 0$. Since the outward unit normal on $AB$ is $\nu = (0,-1)$, we have $\partial_\nu z = -z_y = 0$ on $AB$.
		\end{itemize}
		
		Assume for the sake of contradiction that $z$ attains a negative minimum in $\overline{\Omega}$. Since $z$ is harmonic, this minimum cannot be in the interior. It also cannot be located on $AC^\circ \cup BC^\circ$ because $z > 0$ there. Thus, using again the boundary traces at the endpoints, the negative minimum must be attained at some point on the open segment $AB^\circ$. However, at such a boundary minimum point, Hopf's lemma requires that the outward normal derivative satisfy $\partial_\nu z < 0$, which contradicts the established homogeneous Neumann condition $\partial_\nu z = 0$. 
		
		Consequently, we must have $z \ge 0$ in $\Omega$. Since $z$ is harmonic and strictly positive on $AC^\circ$, it cannot be identically zero. The strong maximum principle then guarantees that
		\[
		z = u_x > 0 \quad \text{in } \Omega.
		\]
		The same boundary-point argument also excludes zeros of $z$ on $AB^\circ$;
		otherwise $z$ would attain a nonnegative boundary minimum there while
		$\partial_\nu z=0$.
		
		Combining Step 2 and Step 3 gives $u_y<0$ and $u_x>0$ in the interior of
		$\Omega$.  The assertion about the maximum follows from these boundary
		monotonicities as well: on $AB$ the function increases in the $x$-direction,
		and on $BC$ the Neumann condition gives
		$u_x=-q u_y>0$, so the tangential derivative in the direction $T=(q,-1)$ is
		$T\cdot\nabla u=-(1+q^2)u_y>0$.  Hence $u$ increases along both Neumann
		sides toward $B=(q,0)$, while $u=0$ on $AC$.  Therefore the maximum on
		$\overline\Omega$ is attained only at $B$.
	\end{proof}

	\section{\texorpdfstring{Monotonicity of \(T(\Omega_q)/q^2\)}{Monotonicity of T(Omega q)/q2}}\label{sec:shape}

	For \(q>0\), set
	\[
	\Omega_q=\{(x,y):0<x<q,\ 0<y<1-x/q\}.
	\]
	The side \(x=0\) carries the Dirichlet condition, while the sides
	\(y=0\) and \(y=1-x/q\) carry Neumann conditions.  Let \(u=u_q\) solve
	\[
	\begin{cases}
		-\Delta u=1 &\text{in }\Omega_q,\\
		u=0 &\text{on }x=0,\\
		\partial_\nu u=0 &\text{on }y=0\text{ and }y=1-x/q.
	\end{cases}
	\]
	We write
	\[
	T(\Omega_q)=\int_{\Omega_q}u\,dx\,dy.
	\]
	\begin{proof}[Proof of Theorem~\ref{thm:intro-mixed-normalized}]
		Consider the deformation
		\[
		(x,y)\longmapsto \left(\frac{q+s}{q}x,y\right),
		\]
		which maps \(\Omega_q\) onto \(\Omega_{q+s}\).  Its velocity field at
		\(s=0\) is
		\[
		\eta(x,y)=\frac1q(x,0).
		\]
		The standard Hadamard formula (see for example \cite{HenrotPierre} or \cite{SokolowskiZolesio}) for the mixed torsion functional gives
		\[
		\frac{d}{dq}T(\Omega_q)
		=
		\int_{\partial\Omega_q}
		\mathcal G(u)\,\eta\cdot\nu\,ds,
		\]
		where
		\[
		\mathcal G(u)=|\nabla u|^2
		\quad\text{on the Dirichlet part},
		\]
		and
		\[
		\mathcal G(u)=2u-|\nabla u|^2
		\quad\text{on the Neumann part}.
		\]
		In the present deformation, the Dirichlet side \(x=0\) is fixed and
		\(\eta\cdot\nu=0\) on \(y=0\).  Hence only the slanted Neumann side
		contributes.  On
		\[
		y=1-x/q,
		\qquad 0<x<q,
		\]
		we have
		\[
		\nu=\frac{(1,q)}{\sqrt{1+q^2}},
		\qquad
		ds=\frac{\sqrt{1+q^2}}{q}\,dx,
		\]
		and therefore
		\[
		\eta\cdot\nu\,ds=\frac{x}{q^2}\,dx.
		\]
		Thus
		\begin{equation}
			\label{eq:Tprime-boundary}
			\frac{d}{dq}T(\Omega_q)
			=
			\frac1{q^2}\int_0^q
			x\left(2u-|\nabla u|^2\right)
			\left(x,1-\frac{x}{q}\right)dx .
		\end{equation}
		
		Set
		\[
		A=\int_{\Omega_q}u_x^2\,dx\,dy,
		\qquad
		B=\int_{\Omega_q}u_y^2\,dx\,dy.
		\]
		Since \(-\Delta u=1\), \(u=0\) on the Dirichlet side, and
		\(\partial_\nu u=0\) on the Neumann sides,
		\begin{equation}
			\label{eq:energy}
			T(\Omega_q)=A+B.
		\end{equation}
		
		Multiplying \(-\Delta u=1\) by \(xu_x\) and integrating by parts gives
		\[
		\int_{\Omega_q}xu_x
		=
		\frac12(A-B)
		+\frac12\int_{\partial\Omega_q}x|\nabla u|^2\nu_x\,ds.
		\]
		On the other hand, the divergence theorem gives
		\[
		\int_{\Omega_q}xu_x
		=
		\int_{\partial\Omega_q}xu\nu_x\,ds
		-T(\Omega_q).
		\]
		Only the slanted side contributes to the boundary integrals, and on that side
		\[
		\nu_x\,ds=\frac1q\,dx.
		\]
		Hence
		\[
		\frac1q\int_0^q
		x\left(2u-|\nabla u|^2\right)
		\left(x,1-\frac{x}{q}\right)dx
		=
		2T(\Omega_q)+A-B.
		\]
		Combining this with \eqref{eq:Tprime-boundary}, we obtain
		\begin{equation}
			\label{eq:Tprime-AB}
			\frac{d}{dq}T(\Omega_q)
			=
			\frac1q\left(2T(\Omega_q)+A-B\right).
		\end{equation}
		Therefore
		\begin{equation}
			\label{eq:normalized-derivative}
			\frac{d}{dq}\left(\frac{T(\Omega_q)}{q^2}\right)
			=
			\frac{A-B}{q^3}.
		\end{equation}
		
		It remains to prove \(A>B\).  Multiplying \(-\Delta u=1\) by \(yu_y\) and
		integrating by parts gives
		\[
		A-B
		=
		\int_{\partial\Omega_q}y|\nabla u|^2\nu_y\,ds
		-2\int_{\Omega_q}yu_y\,dx\,dy.
		\]
		The boundary integral is nonnegative: on \(y=0\) one has \(y=0\), on
		\(x=0\) one has \(\nu_y=0\), and on the slanted side one has
		\(y>0\) and \(\nu_y>0\).  By Proposition~\ref{thm:intro-mixed-gradient},
		\[
		u_y<0\qquad\text{in }\Omega_q.
		\]
		Consequently
		\[
		-2\int_{\Omega_q}yu_y\,dx\,dy>0,
		\]
		and hence
		\[
		A-B>0.
		\]
		By \eqref{eq:normalized-derivative},
		\[
		\frac{d}{dq}\left(\frac{T(\Omega_q)}{q^2}\right)>0.
		\]
		This proves the theorem.
	\end{proof}

	\section{\texorpdfstring{Large \(q\) asymptotics of \(T(\Omega_q)/q^2\)}{Large q asymptotics of T(Omega q)/q2}}\label{sec:large-q}
	
	Define
	\[
	F(q):=\frac{T(\Omega_q)}{q^2}.
	\]
	Let
	\[
	\alpha=\arctan\frac1q,
	\qquad
	q=\cot\alpha.
	\]
	After translating the vertex \(B=(q,0)\) to the origin, scaling by \(q^{-1}\),
	and reflecting across the horizontal Neumann side, the mixed triangle becomes the
	sector-like mixed domain
	\[
	S_\alpha
	=
	\{(r,\theta):|\theta|<\alpha,\ 0<r<\sec\theta\},
	\]
	with homogeneous Neumann conditions on \(\theta=\pm\alpha\) and a Dirichlet
	condition on \(r=\sec\theta\).  Since torsional rigidity scales like the
	fourth power of length and reflection doubles the rigidity,
	\begin{equation}\label{eq:F-sector-relation}
		F(\cot\alpha)
		=
		\frac{\cot^2\alpha}{2}T(S_\alpha).
	\end{equation}
	
	We compare \(S_\alpha\) with the unit mixed sector
	\[
	C_\alpha=\{(r,\theta):|\theta|<\alpha,\ 0<r<1\}.
	\]
	The torsion function in \(C_\alpha\) is \(u_0(r)=(1-r^2)/4\), and
	\(T(C_\alpha)=\alpha/8\).  Put
	\[
	h(\theta)=\sec\theta-1.
	\]
	We shall use the following sector perturbation estimate.
	
	\begin{lemma}\label{lem:sector-expansion}
		As \(\alpha\downarrow0\),
		\[
		T(S_\alpha)
		=
		\frac{\alpha}{8}+\frac{\alpha^3}{12}
		-
		\frac{\zeta(3)}{\pi^3}\alpha^4+O(\alpha^5).
		\]
	\end{lemma}
	
\begin{proof}
	Write
	\[
	R(\theta)=\sec\theta,
	\qquad h(\theta)=R(\theta)-1.
	\]
	Then, uniformly for \(|\theta|\le \alpha\),
	\[
	h(\theta)=\frac{\theta^2}{2}+O(\theta^4),
	\qquad h'(\theta)=\theta+O(\theta^3).
	\]
	In particular,
	\[
	\|h\|_{L^\infty(-\alpha,\alpha)}=O(\alpha^2),
	\qquad
	\int_{-\alpha}^{\alpha}h(\theta)^2\,d\theta=O(\alpha^5).
	\]
	
	Let
	\[
	u_0(r)=\frac{1-r^2}{4}.
	\]
	Then \(-\Delta u_0=1\), and \(u_0\) satisfies the homogeneous Neumann
	condition on the two radial sides. If \(u\) is the torsion function in
	\(S_\alpha\), set
	\[
	w=u-u_0.
	\]
	Then \(w\) is harmonic in \(S_\alpha\), satisfies homogeneous Neumann
	conditions on \(\theta=\pm\alpha\), and on the outer boundary
	\(r=R(\theta)\) it has boundary value
	\[
	w=g,
	\qquad
	g(\theta)=-u_0(R(\theta))
	=\frac{R(\theta)^2-1}{4}
	=\frac{h(\theta)}{2}+\frac{h(\theta)^2}{4}.
	\]
	
	We shall use a simple energy comparison. For boundary data \(f\) on the
	outer side, denote by \(E_{S_\alpha}(f)\) the Dirichlet energy of the
	harmonic function in \(S_\alpha\) with boundary value \(f\) on
	\(r=R(\theta)\) and homogeneous Neumann conditions on the radial sides.
	Similarly, let \(E_{C_\alpha}(f)\) be the corresponding energy in
	\[
	C_\alpha=\{(r,\theta):|\theta|<\alpha,\ 0<r<1\}.
	\]
	Let \(\Lambda_\alpha\) be the corresponding Dirichlet--to--Neumann
	operator on \(r=1\). Thus
	\[
	E_{C_\alpha}(f)=\langle f,\Lambda_\alpha f\rangle .
	\]
	Equivalently, if
	\[
	f(\theta)=a_0+\sum_{m=1}^{\infty}
	a_m\cos\left(\frac{m\pi\theta}{\alpha}\right)
	\]
	is even, then
	\[
	E_{C_\alpha}(f)
	=
	\sum_{m=1}^{\infty}
	\frac{m\pi}{\alpha}\,\alpha a_m^2 .
	\]
	In particular, if
	\[
	f_\alpha(\theta)=\alpha^k F_\alpha(\theta/\alpha)
	\]
	and \(F_\alpha\) remains bounded in \(H^{1/2}(-1,1)\), then
	\[
	E_{C_\alpha}(f_\alpha)=O(\alpha^{2k}).
	\]
	
	Under the change of variables
	\[
	r=R(\theta)\rho ,
	\]
	the energy in \(S_\alpha\) becomes, with
	\[
	a(\theta)=\frac{R'(\theta)}{R(\theta)}=\tan\theta ,
	\]
	\[
	\int_{C_\alpha}
	\left\{
	\rho V_\rho^2+
	\frac1\rho
	\left(V_\theta-\rho a(\theta)V_\rho\right)^2
	\right\}
	d\rho\,d\theta .
	\]
	Since
	\[
	\|a\|_{L^\infty(-\alpha,\alpha)}=O(\alpha),
	\]
	this quadratic form differs from the unit-sector energy
	\[
	\int_{C_\alpha}
	\left(
	\rho V_\rho^2+\frac1\rho V_\theta^2
	\right)
	d\rho\,d\theta
	\]
	by at most \(C\alpha\) times the unit-sector energy. Taking minima gives
	\[
	E_{S_\alpha}(f)
	=
	E_{C_\alpha}(f)+O\bigl(\alpha E_{C_\alpha}(f)\bigr).
	\]
	For \(f=g\), the preceding scaling estimate gives
	\[
	E_{C_\alpha}(g)=O(\alpha^4).
	\]
	Therefore
	\[
	E_{S_\alpha}(g)=E_{C_\alpha}(g)+O(\alpha^5).
	\]
	Moreover,
	\[
	g-\frac h2=\frac{h^2}{4}.
	\]
	Since the \(\Lambda_\alpha\)-energy of \(h^2\) is \(O(\alpha^8)\), the
	quadraticity of \(E_{C_\alpha}\) gives
	\[
	E_{C_\alpha}(g)
	=
	E_{C_\alpha}\left(\frac h2\right)+O(\alpha^6)
	=
	\frac14\langle h,\Lambda_\alpha h\rangle+O(\alpha^6).
	\]
	Consequently,
	\begin{equation}\label{eq:sector-energy-comparison}
		E_{S_\alpha}(g)
		=
		\frac14\langle h,\Lambda_\alpha h\rangle+O(\alpha^5).
	\end{equation}
	
	We next compute \(\int_{S_\alpha}w\,dA\). Since \(w\) is harmonic and
	\(-\Delta u_0=1\), Green's identity gives
	\[
	\int_{S_\alpha} w\,dA
	=
	-\int_{\partial S_\alpha} w\,\partial_\nu u_0\,ds
	-
	\int_{S_\alpha}|\nabla w|^2\,dA .
	\]
	The radial sides give no contribution. On \(r=R(\theta)\),
	\[
	-\partial_\nu u_0\,ds
	=
	\frac{R(\theta)^2}{2}\,d\theta .
	\]
	Therefore
	\[
	\int_{S_\alpha} w\,dA
	=
	\frac12\int_{-\alpha}^{\alpha}g(\theta)R(\theta)^2\,d\theta
	-
	E_{S_\alpha}(g).
	\]
	Because
	\[
	g=\frac h2+O(h^2),
	\qquad
	R^2=1+O(h),
	\qquad
	\int_{-\alpha}^{\alpha}h^2\,d\theta=O(\alpha^5),
	\]
	we obtain
	\[
	\frac12\int_{-\alpha}^{\alpha}g(\theta)R(\theta)^2\,d\theta
	=
	\frac14\int_{-\alpha}^{\alpha}h(\theta)\,d\theta+O(\alpha^5).
	\]
	Using \eqref{eq:sector-energy-comparison},
	\begin{equation}\label{eq:int-w-sector}
		\int_{S_\alpha} w\,dA
		=
		\frac14\int_{-\alpha}^{\alpha}h(\theta)\,d\theta
		-
		\frac14\langle h,\Lambda_\alpha h\rangle
		+
		O(\alpha^5).
	\end{equation}
	
	Also,
	\[
	\int_{S_\alpha}u_0\,dA
	=
	\int_{-\alpha}^{\alpha}\int_0^{R(\theta)}
	\frac{1-r^2}{4}\,r\,dr\,d\theta
	=
	\frac{\alpha}{8}+O(\alpha^5).
	\]
	Combining this with \eqref{eq:int-w-sector} yields
	\begin{equation}\label{eq:sector-pre-final}
		T(S_\alpha)
		=
		\frac{\alpha}{8}
		+
		\frac14\int_{-\alpha}^{\alpha}h(\theta)\,d\theta
		-
		\frac14\langle h,\Lambda_\alpha h\rangle
		+
		O(\alpha^5).
	\end{equation}
	
	The first one-dimensional term is
	\[
	\frac14\int_{-\alpha}^{\alpha}h(\theta)\,d\theta
	=
	\frac14\int_{-\alpha}^{\alpha}
	\left(\frac{\theta^2}{2}+O(\theta^4)\right)d\theta
	=
	\frac{\alpha^3}{12}+O(\alpha^5).
	\]
	
	It remains to evaluate the Dirichlet--to--Neumann term. Since
	\[
	h(\theta)-\frac{\theta^2}{2}=O(\theta^4),
	\]
	the same energy scaling estimate gives
	\[
	\langle h,\Lambda_\alpha h\rangle
	=
	\left\langle \frac{\theta^2}{2},
	\Lambda_\alpha\left(\frac{\theta^2}{2}\right)\right\rangle
	+
	O(\alpha^6).
	\]
	On even functions,
	\[
	\Lambda_\alpha
	\cos\left(\frac{m\pi\theta}{\alpha}\right)
	=
	\frac{m\pi}{\alpha}
	\cos\left(\frac{m\pi\theta}{\alpha}\right),
	\qquad m\ge 1.
	\]
	Furthermore,
	\[
	\frac{\theta^2}{2}
	=
	\alpha^2
	\left[
	\frac16+\frac{2}{\pi^2}
	\sum_{m=1}^\infty
	\frac{(-1)^m}{m^2}
	\cos\left(\frac{m\pi\theta}{\alpha}\right)
	\right].
	\]
	The constant mode is killed by \(\Lambda_\alpha\). Hence Parseval's identity gives
	\[
	\begin{aligned}
		\left\langle \frac{\theta^2}{2},
		\Lambda_\alpha\left(\frac{\theta^2}{2}\right)\right\rangle
		&=
		\sum_{m=1}^\infty
		\frac{m\pi}{\alpha}\,\alpha
		\left(\frac{2\alpha^2}{\pi^2m^2}\right)^2  \\
		&=
		\frac{4\zeta(3)}{\pi^3}\alpha^4.
	\end{aligned}
	\]
	Therefore
	\[
	\langle h,\Lambda_\alpha h\rangle
	=
	\frac{4\zeta(3)}{\pi^3}\alpha^4+O(\alpha^6).
	\]
	Substituting this and the estimate for \(\int h\) into
	\eqref{eq:sector-pre-final}, we obtain
	\[
	T(S_\alpha)
	=
	\frac{\alpha}{8}
	+
	\frac{\alpha^3}{12}
	-
	\frac{\zeta(3)}{\pi^3}\alpha^4
	+
	O(\alpha^5).
	\]
	This proves the lemma.
\end{proof}
	
	Combining Lemma~\ref{lem:sector-expansion} with \eqref{eq:F-sector-relation}
	and the expansion \(\cot^2\alpha=\alpha^{-2}-2/3+O(\alpha^2)\), we obtain
	\begin{equation}\label{eq:F-alpha-expansion}
		F(\cot\alpha)
		=
		\frac{1}{16\alpha}
		-
		\frac{\zeta(3)}{2\pi^3}\alpha^2
		+
		O(\alpha^3).
	\end{equation}

	The preceding calculation is enough to obtain the leading large-\(q\)
	behavior of the original triangular problem.
	
	\begin{proof}[Proof of Theorem~\ref{thm:intro-triangle-asymptotics}]
		The sector calculation gives
		\[
		F(\cot\alpha)
		=
		\frac{1}{16\alpha}
		-
		\frac{\zeta(3)}{2\pi^3}\alpha^2
		+
		O(\alpha^3)
		\qquad (\alpha\downarrow0).
		\]
		Since \(\alpha=\arctan(1/q)\), we have
		\[
		\frac1\alpha=q+\frac1{3q}+O(q^{-3}),
		\qquad
		\alpha^2=q^{-2}+O(q^{-4}).
		\]
		Substitution yields the stated expansion for \(F(q)=T(\Omega_q)/q^2\),
		and multiplying by \(q^2\) gives the equivalent expansion for
		\(T(\Omega_q)\).
	\end{proof}
	
	\section{\texorpdfstring{Regular polygons and the cell reduction}{Regular polygons and the cell reduction}}\label{sec:polygons}
	
	Let \(P_N\) denote the regular \(N\)-gon with area \(\pi\).  Put
	\[
	\alpha=\frac{\pi}{N},
	\qquad
	q=\cot\alpha.
	\]
	The regular \(N\)-gon with side length \(2\) decomposes into \(2N\) copies
	of the triangle \(\Omega_q\).  Its area is \(Nq\).  Scaling this polygon to
	area \(\pi\) gives the exact relation
	\begin{equation}\label{eq:polygon-cell-relation}
		T(P_N)
		=
		2\pi\alpha\,F(\cot\alpha),
		\qquad
		\alpha=\frac{\pi}{N}.
	\end{equation}
	This relation is the geometric reason why the mixed problem on
	\(\Omega_q\) enters the regular-polygon torsion problem.
	
	\begin{proposition}[Cell reduction and leading polygonal asymptotics]\label{prop:cell-leading}
		As \(N\to\infty\),
		\[
		T(P_N)
		=
		\frac{\pi}{8}
		-
		\frac{\pi\zeta(3)}{N^3}
		+
		O(N^{-4}).
		\]
	\end{proposition}
	
	\begin{proof}
		Using \eqref{eq:polygon-cell-relation} and the expansion obtained in
		Section~\ref{sec:large-q}, namely
		\[
		F(\cot\alpha)
		=
		\frac{1}{16\alpha}
		-
		\frac{\zeta(3)}{2\pi^3}\alpha^2
		+
		O(\alpha^3),
		\qquad \alpha\downarrow0,
		\]
		we obtain
		\[
		T(P_N)
		=
		\frac{\pi}{8}
		-
		\frac{\zeta(3)}{\pi^2}\alpha^3
		+
		O(\alpha^4)
		=
		\frac{\pi}{8}
		-
		\frac{\pi\zeta(3)}{N^3}
		+
		O(N^{-4}),
		\]
		because \(\alpha=\pi/N\).  This proves the proposition.
	\end{proof}
	
	\begin{remark}
		Proposition~\ref{prop:cell-leading} is the part of the polygonal asymptotics
		which comes directly from the mixed triangle.  It does not determine the
		coefficient of \(N^{-4}\).  Section~\ref{sec:analytic-N5} is therefore not
		merely a second derivation of the same result.  Its two additional functions
		are to compute the coefficient \(\pi^5/45\) and to provide direct
		derivative estimates for the exact analytic series.  These estimates are
		what make the effective monotonicity criterion in
		Theorem~\ref{thm:intro-polygon-monotonicity} possible.
	\end{remark}
	
	\begin{remark}
		The disk of area \(\pi\) has torsional rigidity \(\pi/8\).  Thus the leading
		term
		\[
		T(P_N)=\frac{\pi}{8}-\frac{\pi\zeta(3)}{N^3}+O(N^{-4})
		\]
		is consistent with the Saint-Venant inequality and implies asymptotic
		monotonicity once a sufficiently sharp remainder estimate is available.
	\end{remark}

	\section{\texorpdfstring{Schwarz--Christoffel--Bergman series and monotonicity of \(T^D(P_N)\)}{Schwarz--Christoffel--Bergman series and monotonicity of TD(PN)}}\label{sec:analytic-N5}
	
	Recall that \(P_N\) denotes the regular \(N\)-gon normalized to have area \(\pi\), and
	let \(T(P_N)\) denote its torsional rigidity:
	\[
	T(P_N)=\int_{P_N}u_N\,dA,
	\qquad
	-\Delta u_N=1\ \text{in }P_N,\quad
	u_N=0\ \text{on }\partial P_N.
	\]
	This section has a different role from the mixed-cell asymptotics in
	Section~\ref{sec:polygons}.  The latter gives a geometric reduction and the
	leading \(N^{-3}\) correction from the triangular problem.  Here we derive an
	exact Schwarz--Christoffel--Bergman series for the regular polygon.  This
	series is used for two purposes which the sector perturbation argument does
	not provide: it identifies the \(N^{-4}\) coefficient in
	Theorem~\ref{thm:intro-polygon-asymptotics}, and it gives an exact analytic
	function of \(t=1/N\).  The proof of
	Theorem~\ref{thm:intro-polygon-monotonicity} then estimates this analytic
	function directly: we prove that \(\cT'(t)<0\) for \(0<t\le1/5\) by elementary
	bounds on the series defining \(S(t)\), \(Q(t)\), and \(M(t)\). 
	
	The equality between torsional rigidity and Bergman analytic content for
	simply connected planar domains was developed by Fleeman--Simanek
	\cite{FleemanSimanek}.  Combined with the classical Schwarz--Christoffel
	representation of regular polygons \cite{DriscollTrefethen,Nehari}, it gives
	a convenient normalized series for \(T(P_N)\).  Related conformal-mapping
	methods for torsion are classical; see, for example,
	\cite{PolyaSzego,Sokolnikoff,FleemanSimanek}.
	
	\subsection{The exact Schwarz--Christoffel--Bergman series}
	
	The purpose of this subsection is to extract from the classical
	Schwarz--Christoffel--Bergman framework an exact series for \(T(P_N)\).
	The Bergman analytic-content identity and its use for computing
	torsional rigidity through conformal maps are due to
	Fleeman--Simanek; see in particular
	\cite[Section 3.3, Theorem 3.5, and formula (8)]{FleemanSimanek}.
	We shall also use the standard conformal covariance of Bergman spaces
	and Bergman projections \cite[Chapter 1]{HedenmalmKorenblumZhu}, the
	monomial Bergman projection formula on the disk
	\cite[Chapter 4]{ZhuBergman}, and the Schwarz--Christoffel
	representation of regular polygons \cite[Chapter 2]{DriscollTrefethen}.
	
	Let \(\Omega\subset\mathbb C\) be a bounded simply connected domain.
	We write
	\[
	A^2(\Omega)
	:=
	\left\{
	g:\Omega\to\mathbb C:
	g \text{ is holomorphic and }
	\int_\Omega |g|^2\,dA<\infty
	\right\}
	\]
	for the Bergman space of \(\Omega\).  If \(P_\Omega\) denotes the
	Bergman projection from \(L^2(\Omega)\) onto \(A^2(\Omega)\), the
	Bergman analytic-content formula gives
	\[
	T(\Omega)
	=
	\frac14
	\inf_{g\in A^2(\Omega)}
	\int_\Omega |\bar z-g(z)|^2\,dA(z).
	\]
	Equivalently,
	\[
	4T(\Omega)
	=
	\|\bar z-P_\Omega(\bar z)\|_{L^2(\Omega)}^2
	=
	\int_\Omega |z|^2\,dA
	-
	\|P_\Omega(\bar z)\|_{L^2(\Omega)}^2 .
	\]
	Thus, for \(\Omega=P_N\), the computation of \(T(P_N)\) reduces to
	computing the Bergman projection of \(\bar z\).
	
	Let
	\[
	\alpha=\frac{\pi}{N}.
	\]
	A Schwarz--Christoffel map from the unit disk \(\mathbb D\) onto a
	regular \(N\)-gon centered at the origin has the form
	\[
	f_N(z)
	=
	C_N\int_0^z(1-w^N)^{-2/N}\,dw,
	\]
	where \(C_N>0\) is chosen so that the image polygon has area \(\pi\).
	Using
	\[
	(1-w^N)^{-2/N}
	=
	\sum_{m\ge0} b_m w^{mN},
	\qquad
	b_m=\frac{(2/N)_m}{m!},
	\]
	where
	\[
	(a)_m=a(a+1)\cdots(a+m-1),\qquad (a)_0=1,
	\]
	we obtain
	\[
	f_N(z)
	=
	C_N\sum_{m\ge0}d_m z^{mN+1},
	\qquad
	d_m
	=
	\frac{(2/N)_m}{m!(mN+1)} .
	\]
	In particular,
	\[
	f_N'(z)
	=
	C_N
	\sum_{m\ge0}(mN+1)d_m z^{mN}.
	\]
	
	The area normalization determines \(C_N\).  Since
	\[
	|P_N|
	=
	\int_{\mathbb D}|f_N'(z)|^2\,dA(z),
	\]
	orthogonality of the monomials on \(\mathbb D\) gives
	\[
	|P_N|
	=
	\pi C_N^2
	\sum_{m\ge0}(mN+1)d_m^2 .
	\]
	Set
	\[
	S_N
	:=
	\sum_{m\ge0}(mN+1)d_m^2 .
	\]
	Because \(|P_N|=\pi\), we have
	\[
	C_N^2=S_N^{-1}.
	\]
	
	We now compute the projection term.  Define
	\[
	U:L^2(P_N)\to L^2(\mathbb D),
	\qquad
	(UH)(z)=H(f_N(z))f_N'(z).
	\]
	Then \(U\) is unitary, and it maps \(A^2(P_N)\) onto \(A^2(\mathbb D)\).
	Consequently the Bergman projections satisfy
	\[
	U P_{P_N}=P_{\mathbb D}U .
	\]
	Applying this identity to \(H(w)=\bar w\), we obtain
	\[
	U(\bar w)(z)=\overline{f_N(z)}\,f_N'(z),
	\]
	and hence
	\[
	\|P_{P_N}(\bar w)\|_{L^2(P_N)}^2
	=
	\left\|
	P_{\mathbb D}\big(\overline{f_N}f_N'\big)
	\right\|_{L^2(\mathbb D)}^2 .
	\]
	
	Using the series for \(f_N\) and \(f_N'\), we have
	\[
	\overline{f_N(z)}\,f_N'(z)
	=
	C_N^2
	\sum_{m,n\ge0}
	(nN+1)d_n d_m\,
	z^{nN}\bar z^{mN+1}.
	\]
	On the unit disk, the Bergman projection of a mixed monomial is
	\[
	P_{\mathbb D}(z^p\bar z^q)
	=
	\begin{cases}
		\displaystyle
		\frac{p-q+1}{p+1}z^{p-q}, & p\ge q,\\[1ex]
		0, & p<q .
	\end{cases}
	\]
	Taking \(p=nN\) and \(q=mN+1\), the nonzero terms occur precisely when
	\(n=m+j\) with \(j\ge1\).  In that case
	\[
	p-q=jN-1,
	\qquad
	p-q+1=jN,
	\]
	and therefore
	\[
	\frac{p-q+1}{p+1}(nN+1)=jN.
	\]
	It follows that
	\[
	P_{\mathbb D}\big(\overline{f_N}f_N'\big)
	=
	C_N^2
	\sum_{j\ge1}
	jN
	\left(\sum_{m\ge0}d_m d_{m+j}\right)
	z^{jN-1}.
	\]
	Since
	\[
	\|z^{jN-1}\|_{L^2(\mathbb D)}^2
	=
	\frac{\pi}{jN},
	\]
	we obtain
	\[
	\begin{aligned}
		\left\|
		P_{\mathbb D}\big(\overline{f_N}f_N'\big)
		\right\|_{L^2(\mathbb D)}^2
		&=
		\pi C_N^4
		\sum_{j\ge1}
		jN
		\left(\sum_{m\ge0}d_m d_{m+j}\right)^2 .
	\end{aligned}
	\]
	Define
	\[
	Q_N
	:=
	\sum_{j\ge1}
	jN
	\left(\sum_{m\ge0}d_m d_{m+j}\right)^2 .
	\]
	Using \(C_N^2=S_N^{-1}\), the projection term is
	\[
	\|P_{P_N}(\bar z)\|_{L^2(P_N)}^2
	=
	\pi S_N^{-2}Q_N .
	\]
	
	It remains to compute the moment
	\[
	M_N:=\int_{P_N}|z|^2\,dA .
	\]
	Let \(\rho\) be the inradius of the area-normalized regular \(N\)-gon.
	The polygon is the union of \(N\) congruent sectors
	\[
	\{(r,\theta):|\theta|<\alpha,\ 0<r<\rho\sec\theta\}.
	\]
	Hence
	\[
	\pi
	=
	|P_N|
	=
	N\int_{-\alpha}^{\alpha}
	\int_0^{\rho\sec\theta} r\,dr\,d\theta
	=
	N\rho^2\tan\alpha.
	\]
	Since \(\alpha=\pi/N\), this gives
	\[
	\rho^2=\frac{\pi}{N\tan\alpha}=\alpha\cot\alpha.
	\]
	Therefore
	\[
	\begin{aligned}
		M_N
		&=
		N\int_{-\alpha}^{\alpha}
		\int_0^{\rho\sec\theta} r^3\,dr\,d\theta  \\
		&=
		\frac{N\rho^4}{4}
		\int_{-\alpha}^{\alpha}\sec^4\theta\,d\theta  \\
		&=
		\frac{N\rho^4}{2}
		\left(\tan\alpha+\frac{\tan^3\alpha}{3}\right)  \\
		&=
		\frac{\pi}{2}
		\left[
		\alpha\cot\alpha+\frac{\alpha\tan\alpha}{3}
		\right].
	\end{aligned}
	\]
	
	Combining the Bergman analytic-content identity with the preceding
	projection computation yields the exact formula
	\begin{align}
		\label{eq:exactT}
		T(P_N)
		=
		\frac14
		\left(M_N-\pi S_N^{-2}Q_N\right),
	\end{align}
	where
	\[
	S_N
	=
	\sum_{m\ge0}(mN+1)d_m^2,
	\qquad
	Q_N
	=
	\sum_{j\ge1}
	jN
	\left(\sum_{m\ge0}d_m d_{m+j}\right)^2 .
	\]

	\subsection{Analytic boundedness and Cauchy coefficient estimates}
	
	We now rewrite the exact formula in the variable
	\[
	t=\frac1N.
	\]
	For complex \(t\) near zero, set
	\[
	d_m(t)
	=
	\frac{(2t)_m}{m!}\frac{t}{m+t},
	\qquad
	d_0(t)=1.
	\]
	We define \(M(t)\), \(S(t)\), and \(Q(t)\) by the series below and then set
	\[
	\cT(t)
	:=
	\frac14
	\left[
	M(t)
	-
	\pi S(t)^{-2}Q(t)
	\right].
	\]
	Thus \(\cT\) is initially defined by the right-hand side of the exact series, not by a
	regular polygon with noninteger number of sides.  When \(t=1/N\), \(N\ge3\) is an
	integer, the exact formula \eqref{eq:exactT} gives
	\[
	\cT(1/N)=T(P_N).
	\]
	Here
	\[
	M(t)
	=
	\frac{\pi}{2}
	\left[
	\pi t\cot(\pi t)
	+
	\frac{\pi t\tan(\pi t)}{3}
	\right],
	\]
	\[
	S(t)
	=
	\sum_{m\ge0}
	\frac{m+t}{t}d_m(t)^2,
	\]
	and
	\[
	Q(t)
	=
	\sum_{j\ge1}
	\frac{j}{t}
	\left(
	\sum_{m\ge0}d_m(t)d_{m+j}(t)
	\right)^2.
	\]
	
	\begin{lemma}\label{lem:analytic-bound}
		The function \(\cT(t)\) is holomorphic in \(|t|<1/4\) and continuous on
		\(|t|\le1/4\), after removing the removable singularity at \(t=0\).  Moreover,
		\[
		|\cT(t)|\le3,
		\qquad |t|\le \frac14 .
		\]
	\end{lemma}
	
	\begin{proof}
		We give a direct estimate from the exact series.  For \(m\ge1\),
		\[
		d_m(t)
		=
		\frac{(2t)_m}{m!}\frac{t}{m+t}.
		\]
		Assume \(|t|\le1/4\).  Then
		\[
		\left|
		\frac{(2t)_m}{m!}
		\right|
		=
		2|t|
		\prod_{k=1}^{m-1}
		\left|
		\frac{k+2t}{k+1}
		\right|
		\le
		2|t|
		\prod_{k=1}^{m-1}
		\frac{k+1/2}{k+1}.
		\]
		Using
		\[
		\prod_{k=1}^{m-1}
		\frac{k+1/2}{k+1}
		=
		\frac{\Gamma(m+1/2)}{\Gamma(3/2)\Gamma(m+1)},
		\]
		Wendel's inequality and \(1/\Gamma(3/2)<6/5\) give
		\[
		\left|
		\frac{(2t)_m}{m!}
		\right|
		\le
		\frac{12}{5}|t|m^{-1/2}.
		\]
		Moreover
		\[
		\left|
		\frac{t}{m+t}
		\right|
		\le
		\frac{|t|}{m-|t|}
		\le
		\frac{4|t|}{3m}.
		\]
		Thus
		\begin{equation}\label{eq:dm-bound-quarter}
			|d_m(t)|
			\le
			\frac{16}{5}|t|^2m^{-3/2},
			\qquad m\ge1,
			\quad |t|\le\frac14 .
		\end{equation}
		Using \eqref{eq:dm-bound-quarter} and
		\((m+|t|)/|t|\le5m/(4|t|)\), we get
		\[
		\begin{aligned}
			|S(t)-1|
			&\le
			\sum_{m\ge1}
			\left|
			\frac{m+t}{t}
			\right|
			|d_m(t)|^2  \\
			&\le
			\frac{64}{5}|t|^3
			\sum_{m\ge1}m^{-2}
			\le
			\frac15\sum_{m\ge1}m^{-2}
			<\frac13 .
		\end{aligned}
		\]
		Consequently
		\[
		|S(t)^{-2}|<\frac94 .
		\]
		Next, for \(j\ge1\), \eqref{eq:dm-bound-quarter} gives
		\[
		\begin{aligned}
			\sum_{m\ge0}|d_m(t)d_{m+j}(t)|
			&\le
			\frac{16}{5}|t|^2j^{-3/2}
			+
			\frac{256}{25}|t|^4
			\sum_{m\ge1}m^{-3/2}(m+j)^{-3/2} \\
			&\le
			\left(
			\frac{16}{5}+\frac{48}{25}
			\right)|t|^2j^{-3/2}
			<
			6|t|^2j^{-3/2}.
		\end{aligned}
		\]
		Here we used \(\sum_{m\ge1}m^{-3/2}<3\) and \(|t|^2\le1/16\).  Therefore
		\[
		\begin{aligned}
			|Q(t)|
			&\le
			\sum_{j\ge1}\frac{j}{|t|}
			\left(6|t|^2j^{-3/2}\right)^2 \\
			&=
			36|t|^3\sum_{j\ge1}j^{-2}
			\le
			\frac{36}{64}\cdot\frac53
			=
			\frac{15}{16}<1.
		\end{aligned}
		\]
		Finally, put \(x=\pi t\).  On \(|x|\le\pi/4\), the standard expansion
		\[
		x\cot x=1-2\sum_{n\ge1}\zeta(2n)\left(\frac{x}{\pi}\right)^{2n}
		\]
		and the bound \(\zeta(2n)<2\) imply
		\[
		|x\cot x|
		\le
		1+4\sum_{n\ge1}4^{-2n}
		=
		\frac{19}{15}.
		\]
		Also, since the Taylor coefficients of \(\tan x\) are positive,
		\[
		|x\tan x|\le |x|\tan |x|\le\frac\pi4<1.
		\]
		Thus
		\[
		|M(t)|
		\le
		\frac{\pi}{2}\left(\frac{19}{15}+\frac13\right)
		=
		\frac{4\pi}{5}
		<3.
		\]
		Combining the estimates for \(M\), \(S\), and \(Q\), we obtain
		\[
		|\cT(t)|
		\le
		\frac14\left[3+\pi\cdot\frac94\cdot1\right]
		<3.
		\]
		The estimates also show uniform convergence of the defining series for \(S\) and
		\(Q\) on \(|t|\le1/4\), after filling in the removable values at \(t=0\).
		Hence \(S\) and \(Q\) are holomorphic in \(|t|<1/4\).  Since
		\(|S(t)-1|<1/3\), the function \(S\) has no zero there, and therefore
		\(S(t)^{-2}\) and \(\cT(t)\) are holomorphic as well.
	\end{proof}
	
	\begin{corollary}\label{cor:cauchy-coeff}
		If
		\[
		\cT(t)=\sum_{k\ge0}A_kt^k
		\]
		is the Taylor expansion at \(t=0\), then
		\[
		|A_k|\le3\cdot4^k,
		\qquad k\ge0.
		\]
	\end{corollary}
	
	\begin{proof}
		This is Cauchy's estimate on the circle \(|t|=1/4\), using
		Lemma~\ref{lem:analytic-bound}.
	\end{proof}
	
	\subsection{The first two nonzero coefficients}
	
	\begin{lemma}\label{lem:first-coefficients}
		As \(t\to0\),
		\[
		\cT(t)
		=
		\frac{\pi}{8}
		-
		\pi\zeta(3)t^3
		+
		\frac{\pi^5}{45}t^4
		+
		O(t^5).
		\]
		Equivalently,
		\[
		A_0=\frac{\pi}{8},
		\qquad
		A_1=A_2=0,
		\qquad
		A_3=-\pi\zeta(3),
		\qquad
		A_4=\frac{\pi^5}{45}.
		\]
	\end{lemma}
	
	\begin{proof}
		For \(m\ge1\),
		\[
		d_m(t)=\frac{(2t)_m}{m!}\frac{t}{m+t}.
		\]
		Since
		\[
		\frac{(2t)_m}{m!}
		=
		\frac{2t}{m}
		\prod_{\ell=1}^{m-1}
		\left(1+\frac{2t}{\ell}\right),
		\]
		we have, for fixed \(m\),
		\[
		\frac{(2t)_m}{m!}
		=
		\frac{2t}{m}
		\left(1+2H_{m-1}t+O_m(t^2)\right),
		\]
		where \(H_m=1+1/2+\cdots+1/m\).  Also
		\[
		\frac{t}{m+t}
		=
		\frac{t}{m}-\frac{t^2}{m^2}+O_m(t^3).
		\]
		Thus
		\begin{equation}\label{eq:dm-small-expansion}
			d_m(t)
			=
			\frac{2t^2}{m^2}
			+
			\left(
			\frac{4H_{m-1}}{m^2}
			-
			\frac{2}{m^3}
			\right)t^3
			+
			O_m(t^4).
		\end{equation}
		The bounds used in Lemma~\ref{lem:analytic-bound}, together with Cauchy's
		estimate on small circles centered at the origin, justify the coefficient
		extractions below.  More explicitly, on every circle \(|t|=r<1/4\), the
		Taylor coefficients of the functions \(d_m(t)\) which are needed up to order
		three are bounded by an absolute multiple of
		\((1+H_m)m^{-2}\).  Thus the coefficient sums which occur in \(S\) and
		\(Q\) are dominated by absolutely summable series such as
		\(\sum m^{-3}H_m\) and \(\sum m^{-4}\).
		
		First,
		\[
		S(t)
		=
		1+s_3t^3+s_4t^4+O(t^5).
		\]
		Substituting \eqref{eq:dm-small-expansion} into
		\[
		S(t)=1+
		\sum_{m\ge1}\frac{m+t}{t}d_m(t)^2
		\]
		gives
		\[
		s_3=4\sum_{m\ge1}\frac1{m^3}=4\zeta(3),
		\]
		and
		\[
		s_4
		=
		16\sum_{m\ge1}\frac{H_{m-1}}{m^3}
		-
		4\sum_{m\ge1}\frac1{m^4}.
		\]
		Using the classical Euler sum identity
		\cite{BorweinBorweinGirgensohn}
		\[
		\sum_{m\ge1}\frac{H_m}{m^3}=\frac54\zeta(4),
		\]
		we find
		\[
		\sum_{m\ge1}\frac{H_{m-1}}{m^3}
		=
		\sum_{m\ge1}\frac{H_m}{m^3}
		-
		\sum_{m\ge1}\frac1{m^4}
		=
		\frac14\zeta(4).
		\]
		Therefore \(s_4=0\), and hence
		\begin{equation}\label{eq:S-expansion-first}
			S(t)=1+4\zeta(3)t^3+O(t^5).
		\end{equation}
		
		Next,
		\[
		Q(t)=q_3t^3+q_4t^4+O(t^5).
		\]
		For fixed \(j\ge1\),
		\[
		\sum_{m\ge0}d_m(t)d_{m+j}(t)
		=
		d_j(t)+O_j(t^4),
		\]
		because \(d_0(t)=1\) and \(d_m(t)=O_m(t^2)\) for \(m\ge1\).  The domination
		just described also applies after summing over \(j\): the relevant coefficient
		bounds are controlled by \(j^{-3}H_j\) and \(j^{-4}\), both of which are
		summable.  Hence the resulting summation over \(j\) is legitimate.  Using
		\eqref{eq:dm-small-expansion}, we obtain
		\[
		q_3
		=
		\sum_{j\ge1}j\left(\frac{2}{j^2}\right)^2
		=
		4\zeta(3),
		\]
		and
		\[
		\begin{aligned}
			q_4
			&=
			\sum_{j\ge1}
			2j\frac{2}{j^2}
			\left(
			\frac{4H_{j-1}}{j^2}
			-
			\frac{2}{j^3}
			\right) \\
			&=
			16\sum_{j\ge1}\frac{H_{j-1}}{j^3}
			-
			8\sum_{j\ge1}\frac1{j^4}
			=
			-4\zeta(4).
		\end{aligned}
		\]
		Thus
		\begin{equation}\label{eq:Q-expansion-first}
			Q(t)=4\zeta(3)t^3-4\zeta(4)t^4+O(t^5).
		\end{equation}
		From \eqref{eq:S-expansion-first},
		\[
		S(t)^{-2}=1+O(t^3),
		\]
		and so
		\begin{equation}\label{eq:S2Q-expansion-first}
			S(t)^{-2}Q(t)
			=
			4\zeta(3)t^3-4\zeta(4)t^4+O(t^5).
		\end{equation}
		
		Finally, putting \(x=\pi t\),
		\[
		x\cot x
		=
		1-\frac{x^2}{3}-\frac{x^4}{45}+O(x^6),
		\]
		and
		\[
		\frac{x\tan x}{3}
		=
		\frac{x^2}{3}+\frac{x^4}{9}+O(x^6).
		\]
		Therefore
		\[
		x\cot x+\frac{x\tan x}{3}
		=
		1+\frac{4x^4}{45}+O(x^6),
		\]
		and hence
		\begin{equation}\label{eq:M-expansion-first}
			M(t)
			=
			\frac{\pi}{2}
			+
			\frac{2\pi^5}{45}t^4
			+
			O(t^6).
		\end{equation}
		Combining \eqref{eq:S2Q-expansion-first} and
		\eqref{eq:M-expansion-first} in
		\[
		\cT(t)=\frac14\left(M(t)-\pi S(t)^{-2}Q(t)\right),
		\]
		we get
		\[
		\cT(t)
		=
		\frac{\pi}{8}
		-
		\pi\zeta(3)t^3
		+
		\left(
		\frac{\pi^5}{90}
		+
		\pi\zeta(4)
		\right)t^4
		+
		O(t^5).
		\]
		Since \(\zeta(4)=\pi^4/90\), the coefficient of \(t^4\) is
		\(\pi^5/45\), as claimed.
	\end{proof}
	
	\begin{proof}[Proof of Theorem~\ref{thm:intro-polygon-asymptotics}]
		By the exact-series definition above, \(\cT(1/N)=T(P_N)\) for every integer
		\(N\ge3\).  Substituting \(t=1/N\) into Lemma~\ref{lem:first-coefficients}
		gives
		\[
		T(P_N)
		=
		\frac{\pi}{8}
		-
		\frac{\pi\zeta(3)}{N^3}
		+
		\frac{\pi^5}{45N^4}
		+
		O(N^{-5}),
		\]
		which is the asserted expansion.  This also refines the leading asymptotics
		obtained from the mixed-cell reduction in Section~\ref{sec:polygons}.
	\end{proof}

	\subsection{A direct derivative estimate on \texorpdfstring{\(0<t\le1/5\)}{0 < t <= 1/5}}
	
	We now prove Theorem~\ref{thm:intro-polygon-monotonicity} without introducing
	any high-order coefficient certificate.  The point is to show directly that the
	analytic function \(\cT\) is strictly decreasing on \((0,1/5]\).  Put
	\[
	R(t):=S(t)^{-2}Q(t),
	\qquad
	\cT(t)=\frac14\bigl(M(t)-\pi R(t)\bigr).
	\]
	
	\begin{lemma}\label{lem:SQ-direct-estimates}
		For \(0<t\le1/5\),
		\[
		1\le S(t)\le \frac{11}{10},
		\qquad
		0<S'(t)\le \frac{275}{6}t^2,
		\]
		and
		\[
		R'(t)\ge \frac{5600}{1089}t^2.
		\]
	\end{lemma}
	
	\begin{proof}
		For \(m\ge1\), the coefficients in the exact series can be written as
		\begin{equation}\label{eq:dm-positive-factorization-N5}
			d_m(t)=
			\frac{2t^2}{m(m+t)}
			\prod_{\ell=1}^{m-1}\left(1+\frac{2t}{\ell}\right).
		\end{equation}
		All factors are positive for \(t>0\).  We shall use the following two-sided
		bounds for the product.  Put \(a=2t\), so that \(0<a\le2/5\).  Since
		\(\log(1+ax)\ge a\log(1+x)\) for \(0\le a\le1\) and \(x\ge0\),
		\[
		\prod_{\ell=1}^{m-1}\left(1+\frac{a}{\ell}\right)
		\ge
		\prod_{\ell=1}^{m-1}\left(1+\frac1\ell\right)^a
		=m^a.
		\]
		For the upper bound we use Wendel's inequality in the form
		\[
		\frac{\Gamma(m+a)}{\Gamma(m)}\le m^a,
		\qquad
		\Gamma(1+a)\ge (1+a)^{a-1},
		\qquad 0\le a\le1.
		\]
		Hence
		\[
		\prod_{\ell=1}^{m-1}\left(1+\frac{a}{\ell}\right)
		=
		\frac{\Gamma(m+a)}{\Gamma(m)\Gamma(1+a)}
		\le m^a(1+a)^{1-a}.
		\]
		The function \((1-a)\log(1+a)\) is increasing on \([0,2/5]\), because
		\(\log(1+a)\le a\le (1-a)/(1+a)\) there.  Therefore
		\[
		(1+a)^{1-a}\le \left(\frac75\right)^{3/5}<\frac54,
		\]
		the last inequality following after raising to the fifth power.  Thus, for
		\(0<t\le1/5\),
		\begin{equation}\label{eq:dm-upper-direct-N5}
			d_m(t)\le \frac52 t^2m^{-2+2t},
		\end{equation}
		and, since \(m+t\le(6/5)m\),
		\begin{equation}\label{eq:dm-lower-direct-N5}
			d_m(t)\ge \frac53 t^2m^{-2+2t}.
		\end{equation}
		
		We also record the justification of termwise differentiation.  Fix a compact
		interval \(I=[\alpha,\beta]\subset(0,1/5]\).  From the logarithmic derivative
		formula
		\begin{equation}\label{eq:log-dm-derivative-N5}
			\frac{d}{dt}\log d_m(t)
			=
			\frac2t-\frac1{m+t}
			+\sum_{\ell=1}^{m-1}\frac2{\ell+2t},
		\end{equation}
		and from \eqref{eq:dm-upper-direct-N5}, there is a constant \(C_I\) such that
		\[
		|d_m(t)|\le C_I m^{-2+2\beta},
		\qquad
		|d_m'(t)|\le C_I(1+\log m)m^{-2+2\beta},
		\qquad t\in I.
		\]
		Since \(-3+4\beta<-2\), the series defining \(S\) and its differentiated
		series are uniformly convergent on \(I\).  Similarly, if
		\[
		B_j(t):=\sum_{m\ge0}d_m(t)d_{m+j}(t),
		\]
		then the series defining \(B_j\) and \(B_j'\) are uniformly convergent on
		\(I\), and
		\[
		|B_j(t)|\le C_Ij^{-2+2\beta},
		\qquad
		|B_j'(t)|\le C_I(1+\log j)j^{-2+2\beta}.
		\]
		Here we used the elementary convolution estimate
		\(\sum_{m\ge1}m^{-\gamma}(m+j)^{-\gamma}\le C_\gamma j^{-\gamma}\),
		and the same estimate with one logarithmic factor, where
		\(\gamma=2-2\beta>1\).  Consequently the series defining \(Q\) and its
		derivative are uniformly convergent on \(I\).  The differentiations below
		are therefore justified on \((0,1/5]\).
		
		Let
		\[
		w_m(t):=\frac{m+t}{t}d_m(t)^2,
		\qquad m\ge1.
		\]
		Then \(S(t)=1+\sum_{m\ge1}w_m(t)\).  From
		\eqref{eq:dm-upper-direct-N5} and \(m+t\le(6/5)m\),
		\begin{equation}\label{eq:w-upper-N5}
			w_m(t)\le \frac{15}{2}t^3m^{-3+4t}.
		\end{equation}
		Set \(p=3-4t\).  Since \(0<t\le1/5\), we have \(p\ge11/5\).  Hence
		\[
		\sum_{m\ge1}m^{-p}
		\le
		1+2^{-p}+\int_2^\infty x^{-p}\,dx
		\le
		1+\frac14+\frac{5}{12}
		=
		\frac53.
		\]
		It follows from \eqref{eq:w-upper-N5} that
		\[
		S(t)\le1+\frac{15}{2}t^3\cdot\frac53
		\le1+\frac1{10}
		=\frac{11}{10}.
		\]
		The lower bound \(S(t)\ge1\) is immediate.
		
		Next,
		\[
		\frac{w_m'(t)}{w_m(t)}
		=
		\frac3t-\frac1{m+t}
		+\sum_{\ell=1}^{m-1}\frac4{\ell+2t}
		\le
		\frac3t+4H_{m-1}.
		\]
		The terms are positive, so \(S'(t)>0\).  Moreover,
		\[
		\sum_{m\ge1}H_{m-1}m^{-p}
		=
		\sum_{\ell\ge1}\frac1\ell\sum_{m\ge \ell+1}m^{-p}
		\le
		\frac1{p-1}\sum_{\ell\ge1}\ell^{-p}
		\le
		\frac56\cdot\frac53
		=
		\frac{25}{18}.
		\]
		Using this, the preceding estimate for \(\sum m^{-p}\), and
		\eqref{eq:w-upper-N5}, we get
		\[
		\begin{aligned}
			S'(t)
			&\le
			\frac{15}{2}t^3
			\left(
			\frac3t\sum_{m\ge1}m^{-p}
			+4\sum_{m\ge1}H_{m-1}m^{-p}
			\right) \\
			&\le
			\frac{15}{2}t^3
			\left(
			\frac5t+\frac{50}{9}
			\right)
			\le
			\frac{275}{6}t^2,
		\end{aligned}
		\]
		where the last step uses \(t\le1/5\).
		
		It remains to bound \(R'\) from below.  Since \(B_j(t)\ge d_j(t)\),
		\eqref{eq:dm-lower-direct-N5} gives
		\[
		\begin{aligned}
			Q(t)
			&=\sum_{j\ge1}\frac jt B_j(t)^2
			\ge
			\sum_{j\ge1}\frac jt d_j(t)^2  \\
			&\ge
			\frac{25}{9}t^3\sum_{j\ge1}j^{-3+4t}
			\ge
			\frac{25}{9}t^3\sum_{j\ge1}j^{-3}
			>
			\frac{10}{3}t^3.
		\end{aligned}
		\]
		Here we used the elementary bound \(\zeta(3)>6/5\), for instance from the
		finite partial sum \(\sum_{j=1}^{16}j^{-3}>6/5\).
		
		By \eqref{eq:log-dm-derivative-N5},
		\[
		\frac{d_m'(t)}{d_m(t)}\ge \frac2t-1,
		\qquad m\ge1.
		\]
		Therefore every summand in \(B_j\) has logarithmic derivative at least
		\(2/t-1\) (for the summand \(m=0\), this is the estimate for \(d_j\)).  Hence
		\[
		B_j'(t)\ge\left(\frac2t-1\right)B_j(t),
		\]
		and so each positive summand \(jB_j(t)^2/t\) of \(Q\) satisfies
		\[
		\frac{d}{dt}\log\left(\frac jt B_j(t)^2\right)
		\ge
		\frac3t-2.
		\]
		Thus
		\[
		\frac{Q'(t)}{Q(t)}\ge \frac3t-2.
		\]
		Since \(R=QS^{-2}\), we have
		\[
		R'(t)=R(t)\left(\frac{Q'(t)}{Q(t)}-2\frac{S'(t)}{S(t)}\right).
		\]
		Using \(S\le11/10\), \(S\ge1\), and the bounds already proved,
		\[
		R(t)\ge \frac{(10/3)t^3}{(11/10)^2}=\frac{1000}{363}t^3,
		\]
		and
		\[
		\frac{Q'(t)}{Q(t)}-2\frac{S'(t)}{S(t)}
		\ge
		\frac3t-2-\frac{275}{3}t^2.
		\]
		Consequently,
		\[
		R'(t)
		\ge
		\frac{1000}{363}t^2
		\left(3-2t-\frac{275}{3}t^3\right).
		\]
		For \(0<t\le1/5\), the factor in parentheses is at least
		\[
		3-\frac25-\frac{275}{3}\cdot\frac1{125}
		=
		\frac{28}{15}.
		\]
		Therefore
		\[
		R'(t)\ge \frac{1000}{363}\cdot\frac{28}{15}t^2
		=\frac{5600}{1089}t^2.
		\]
		This proves the lemma.
	\end{proof}
	
	\begin{lemma}\label{lem:M-derivative-direct}
		For \(0<t\le1/5\),
		\[
		M'(t)<16t^2.
		\]
	\end{lemma}
	
	\begin{proof}
		Put \(x=\pi t\) and
		\[
		F(x):=x\cot x+\frac{x\tan x}{3}.
		\]
		Then \(M(t)=(\pi/2)F(\pi t)\).  Since \(t\le1/5\) and \(\pi<22/7\), we have
		\(0<x\le22/35\).
		
		We first bound the cotangent part.  The standard expansion
		\[
		x\cot x=1-2\sum_{n\ge1}\zeta(2n)\left(\frac{x}{\pi}\right)^{2n}
		\]
		shows that
		\begin{equation}\label{eq:xcot-derivative-N5}
			(x\cot x)'
			\le
			-\frac{2x}{3}-\frac{4x^3}{45}.
		\end{equation}
		Next we prove a simple bound for the secant.  Let \(y=x^2\).  Since
		\(x\le22/35\), we have \(0\le y<2/5\).  The alternating Taylor expansion of
		\(\cos(2x)\) gives
		\[
		\cos^2x=\frac{1+\cos(2x)}2
		\ge
		1-y+\frac{y^2}{3}-\frac{2y^3}{45}.
		\]
		Moreover,
		\[
		\left(1-y+\frac{y^2}{3}-\frac{2y^3}{45}\right)(1+y+y^2)-1
		=
		\frac{y^2}{45}\left(15-32y+13y^2-2y^3\right)\ge0,
		\]
		because \(0\le y\le2/5\).  Hence
		\begin{equation}\label{eq:sec-square-N5}
			\sec^2x\le1+x^2+x^4,
			\qquad 0\le x\le \frac{22}{35}.
		\end{equation}
		It follows that
		\[
		\tan x\le x+\frac{x^3}{3}+\frac{x^5}{5},
		\qquad
		x\sec^2x\le x+x^3+x^5.
		\]
		Combining these estimates with \eqref{eq:xcot-derivative-N5}, we obtain
		\[
		\begin{aligned}
			F'(x)
			&=(x\cot x)'+\frac13(\tan x+x\sec^2x) \\
			&\le
			-\frac{2x}{3}-\frac{4x^3}{45}
			+\frac13\left(2x+\frac{4x^3}{3}+\frac{6x^5}{5}\right) \\
			&=
			\frac{16}{45}x^3+\frac25x^5.
		\end{aligned}
		\]
		Since \(x\le22/35\),
		\[
		F'(x)
		\le
		\left(\frac{16}{45}+\frac25\left(\frac{22}{35}\right)^2\right)x^3
		=
		\frac{28312}{55125}x^3.
		\]
		Therefore, using again \(t\le1/5\) and \(\pi<22/7\),
		\[
		M'(t)
		=\frac{\pi^2}{2}F'(\pi t)
		\le
		\frac{28312}{55125}\frac{\pi^5}{2}t^3
		\le
		\frac{28312}{55125}\frac{1}{10}\left(\frac{22}{7}\right)^5t^2
		<16t^2.
		\]
	\end{proof}
	
	\subsection{Proof of monotonicity for \texorpdfstring{\(N\ge3\)}{N >= 3}}
	
	\begin{proof}[Proof of Theorem~\ref{thm:intro-polygon-monotonicity}]
		For \(0<t\le1/5\), Lemmas~\ref{lem:SQ-direct-estimates} and
		\ref{lem:M-derivative-direct} give
		\[
		\cT'(t)
		=
		\frac14\bigl(M'(t)-\pi R'(t)\bigr)
		<
		\frac14\left(16-\pi\frac{5600}{1089}\right)t^2.
		\]
		Using the elementary lower bound \(\pi>157/50\),
		\[
		\pi\frac{5600}{1089}-16
		>
		\frac{157}{50}\frac{5600}{1089}-16
		=
		\frac{160}{1089}>0.
		\]
		Thus \(\cT'(t)<0\) on \((0,1/5]\), and \(\cT\) is strictly decreasing there.
		
		If \(N\ge5\), then
		\[
		0<\frac1{N+1}<\frac1N\le\frac15.
		\]
		Since \(\cT(1/N)=T(P_N)\),
		\[
		T(P_{N+1})-T(P_N)
		=
		\cT\left(\frac1{N+1}\right)-\cT\left(\frac1N\right)
		=
		-\int_{1/(N+1)}^{1/N}\cT'(t)\,dt
		>0.
		\]

		We now verify directly the two remaining inequalities
		\[
		T(P_3)<T(P_4)<T(P_5).
		\]
		Together with the monotonicity already proved for \(N\ge 5\), this gives the
		full monotonicity of \(T(P_N)\) for all \(N\ge 3\).
		
		We first consider the equilateral triangle.  If the side length is \(a\), and the
		triangle has vertices \((0,0),(a,0),(a/2,\sqrt3 a/2)\), then
		\[
		u(x,y)=
		\frac{1}{2\sqrt3\,a}\,
		y(\sqrt3 x-y)(\sqrt3(a-x)-y)
		\]
		satisfies
		\[
		-\Delta u=1,\qquad u=0\quad\hbox{on the boundary}.
		\]
		A direct integration gives
		\[
		T=\int u
		=
		\frac{\sqrt3}{320}a^4.
		\]
		For the area-normalized equilateral triangle, \(a^2=4\pi/\sqrt3\). Hence
		\[
		T(P_3)=\frac{\sqrt3\,\pi^2}{60}.
		\]
		
		For the area-normalized square, the standard rectangular torsion formula gives
		\[
		T(P_4)
		=
		\frac{\pi^2}{12}
		\left(
		1-\frac{192}{\pi^5}
		\sum_{\ell=0}^{\infty}
		\frac{\tanh((2\ell+1)\pi/2)}{(2\ell+1)^5}
		\right).
		\]
		Since \(\tanh x<1\),
		\[
		T(P_4)
		>
		\frac{\pi^2}{12}
		\left(
		1-\frac{192}{\pi^5}
		\sum_{\ell=0}^{\infty}\frac{1}{(2\ell+1)^5}
		\right)
		=
		\frac{\pi^2}{12}
		\left(
		1-\frac{186\zeta(5)}{\pi^5}
		\right).
		\]
		We use the elementary bound
		\[
		\zeta(5)
		<
		1+\frac1{2^5}+\int_2^\infty x^{-5}\,dx
		=
		\frac{67}{64}.
		\]
		Together with \(\pi>157/50\), this gives
		\[
		\frac{186\zeta(5)}{\pi^5}
		<
		\frac{186\cdot 67}{64}\left(\frac{50}{157}\right)^5
		<
		\frac{13}{20}.
		\]
		Therefore
		\[
		T(P_4)>\frac{\pi^2}{12}\cdot\frac7{20}
		=
		\frac{7\pi^2}{240}.
		\]
		Since \(\sqrt3<7/4\), we have
		\[
		T(P_3)=\frac{\sqrt3\,\pi^2}{60}
		<
		\frac{7\pi^2}{240}
		<
		T(P_4).
		\]
		
		It remains to prove \(T(P_4)<T(P_5)\).  We first obtain an upper bound for
		\(T(P_4)\).  Keeping only the first positive term in the square series gives
		\[
		T(P_4)
		<
		\frac{\pi^2}{12}
		\left(
		1-\frac{192}{\pi^5}\tanh\frac{\pi}{2}
		\right).
		\]
		Since
		\[
		\tanh\frac{\pi}{2}
		=
		\frac{e^\pi-1}{e^\pi+1},
		\]
		and since \(\pi>157/50\) implies
		\[
		e^\pi>
		\sum_{k=0}^{9}\frac{(157/50)^k}{k!}
		>23,
		\]
		we get
		\[
		\tanh\frac{\pi}{2}>\frac{11}{12}.
		\]
		Using also \(\pi<355/113\), we find
		\[
		T(P_4)
		<
		\frac{(355/113)^2}{12}
		\left(
		1-\frac{176}{(355/113)^5}
		\right)
		<
		\frac7{20}.
		\]
		
		We now prove that \(T(P_5)>7/20\).  We use the exact
		Schwarz--Christoffel/Bergman formula as before:
		\[
		T(P_5)
		=
		\frac14
		\left(
		M_5-\pi S_5^{-2}Q_5
		\right),
		\]
		where
		\[
		M_5=
		\frac{\pi}{2}
		\left(
		\frac{\pi}{5}\cot\frac{\pi}{5}
		+
		\frac{\pi}{15}\tan\frac{\pi}{5}
		\right)
		=
		\frac{\pi^2}{10}
		\left(
		\cot\frac{\pi}{5}
		+
		\frac13\tan\frac{\pi}{5}
		\right),
		\]
		and
		\[
		S_5=\sum_{m\ge0}(5m+1)d_m^2,\qquad
		Q_5=
		\sum_{j\ge1}5j
		\left(
		\sum_{m\ge0}d_m d_{m+j}
		\right)^2,
		\]
		with
		\[
		d_m=\frac{(2/5)_m}{m!(5m+1)}.
		\]
		Since \(S_5\ge1\), it is enough to bound \(M_5\) from below and \(Q_5\) from
		above.
		
		Let
		\[
		x=\tan\frac{\pi}{5}.
		\]
		Then
		\[
		x^2=5-2\sqrt5.
		\]
		Since \(\sqrt5>502/225\), we get
		\[
		x^2<\frac{121}{225},
		\qquad\hbox{hence}\qquad
		x<\frac{11}{15}.
		\]
		The function \(x\mapsto x^{-1}+x/3\) is decreasing on \((0,1)\). Therefore
		\[
		\cot\frac{\pi}{5}
		+
		\frac13\tan\frac{\pi}{5}
		=
		\frac1x+\frac{x}{3}
		>
		\frac{15}{11}+\frac{11}{45}
		=
		\frac{796}{495}
		>
		\frac85.
		\]
		Using \(\pi>157/50\), we obtain
		\[
		M_5
		>
		\frac{(157/50)^2}{10}\cdot\frac85
		=
		\frac{24649}{15625}.
		\]
		
		It remains to show that \(Q_5<1/20\).  For \(m\ge1\), the coefficients satisfy
		\[
		d_m
		=
		\frac{2}{25}\frac{1}{m(m+1/5)}
		\prod_{\ell=1}^{m-1}
		\left(1+\frac{2}{5\ell}\right).
		\]
		Using \(\log(1+y)\le y\), \(H_{m-1}\le 1+\log m\), and
		\(e^{2/5}<3/2\), we get
		\[
		d_m
		\le
		\frac{3}{25}m^{-8/5},
		\qquad m\ge1.
		\]
		Also \(d_m\) is decreasing, because
		\[
		\frac{d_{m+1}}{d_m}
		=
		\frac{m+2/5}{m+1}\frac{5m+1}{5m+6}
		<1.
		\]
		Set
		\[
		D:=\sum_{m\ge1}d_m.
		\]
		A direct rational calculation gives
		\[
		\sum_{m=1}^{10}d_m<\frac{11}{75}.
		\]
		The tail estimate above gives
		\[
		\sum_{m\ge11}d_m
		\le
		\frac3{25}\int_{10}^{\infty}x^{-8/5}\,dx
		=
		\frac15\,10^{-3/5}
		<
		\frac4{75}.
		\]
		Hence
		\[
		D<\frac15.
		\]
		
		For
		\[
		B_j:=\sum_{m\ge0}d_m d_{m+j},
		\]
		monotonicity of \(d_m\) gives
		\[
		B_j
		=
		d_j+\sum_{m\ge1}d_m d_{m+j}
		\le
		d_j+d_{j+1}D
		<
		d_j+\frac15d_{j+1}.
		\]
		Therefore
		\[
		Q_5
		=
		5\sum_{j\ge1}jB_j^2
		<
		5\sum_{j\ge1}j
		\left(d_j+\frac15d_{j+1}\right)^2.
		\]
		For the first ten terms, direct rational arithmetic gives
		\[
		5\sum_{j=1}^{10}j
		\left(d_j+\frac15d_{j+1}\right)^2
		=
		\frac{
			5311812193770392187146788036183
		}{
			120959437643980979919433593750000
		}
		<
		\frac{11}{250}.
		\]
		For the tail, using \(d_j\le (3/25)j^{-8/5}\), we have
		\[
		d_j+\frac15d_{j+1}
		\le
		\frac65d_j
		\le
		\frac{18}{125}j^{-8/5}.
		\]
		Thus
		\[
		5\sum_{j\ge11}j
		\left(d_j+\frac15d_{j+1}\right)^2
		\le
		5\left(\frac{18}{125}\right)^2
		\int_{10}^{\infty}x^{-11/5}\,dx
		=
		\frac{54}{625}10^{-6/5}
		<
		\frac{18}{3125}.
		\]
		Consequently
		\[
		Q_5
		<
		\frac{11}{250}+\frac{18}{3125}
		<
		\frac1{20}.
		\]
		
		Combining the estimates for \(M_5\) and \(Q_5\), and using \(\pi<22/7\), we obtain
		\[
		T(P_5)
		\ge
		\frac14(M_5-\pi Q_5)
		>
		\frac14
		\left(
		\frac{24649}{15625}
		-
		\frac{22}{7}\cdot\frac1{20}
		\right)
		=
		\frac{310711}{875000}
		>
		\frac7{20}.
		\]
		Since \(T(P_4)<7/20<T(P_5)\), we have proved
		\[
		T(P_4)<T(P_5).
		\]
		Together with the first inequality, this gives
		\[
		T(P_3)<T(P_4)<T(P_5).
		\]
		Therefore, combined with the already established monotonicity for all
		\(N\ge5\), the sequence \(T(P_N)\) is strictly increasing for every
		integer \(N\ge3\).
	\end{proof}

	\section{Monotonicity for the first mixed eigenvalue on \texorpdfstring{$\Omega_q$}{Omega q}}
	\label{secmonoeigenvalue}
	
	Let
	\[
	\Omega_q=\{(x,y):0<x<q,\ 0<y<1-x/q\},\qquad q>0,
	\]
	with vertices
	\[
	A=(0,0),\qquad B=(q,0),\qquad C=(0,1).
	\]
	We impose a Dirichlet condition on the vertical side \(AC\) and Neumann conditions
	on the other two sides. Denote by \(\mu(q)\) the first eigenvalue of the mixed
	problem
	\[
	\begin{cases}
		-\Delta u=\mu(q)u & \text{in }\Omega_q,\\
		u=0 & \text{on }AC,\\
		u_y=0 & \text{on }AB,\\
		u_x+q u_y=0 & \text{on }BC.
	\end{cases}
	\]
	Equivalently,
	\[
	\mu(q)
	=
	\min_{\substack{v\in H^1(\Omega_q)\\ v=0\text{ on }AC}}
	\frac{\displaystyle\int_{\Omega_q}|\nabla v|^2}
	{\displaystyle\int_{\Omega_q}v^2}.
	\]
	The first eigenvalue is simple, and the corresponding eigenfunction can be chosen
	strictly positive in \(\Omega_q\). We denote this positive eigenfunction by \(u_q\). The following monotonicity follows from Jerison–Nadirashvili \cite[Theorem 1.1]{JerisonNadirashvili}. Indeed, reflect $\Omega_q$ oddly across the Dirichlet side $AC$ and evenly across the Neumann side $AB$ leads to the monotonicity result in the right triangle for the mixed condition.
	
	\begin{theorem}[Directional monotonicity of the first eigenfunction]
		\label{monoeigenfunction}
		For every \(q>0\), the positive first mixed eigenfunction \(u_q\) satisfies
		\[
		\partial_y u_q<0
		\qquad\text{and}\qquad
		\partial_x u_q>0
		\qquad\text{in }\Omega_q .
		\]
	\end{theorem}
	
	In the following, we give a separate proof, without using the continuity method as in \cite{JerisonNadirashvili}.
	\begin{proof}
		We write \(u=u_q\) and \(\mu=\mu(q)\). Let
		\[
		w=u_y .
		\]
		Differentiating the eigenvalue equation gives
		\[
		-\Delta w=\mu w
		\qquad\text{in }\Omega_q .
		\]
		On \(AB\), the Neumann condition gives \(w=0\). On \(AC\), since \(u=0\) and
		\(y\) is tangential to \(AC\), we also have \(w=u_y=0\).
		
		It remains to understand the boundary behavior of \(w\) on the slanted side \(BC\).
		Let
		\[
		N=(1,q),\qquad T=(q,-1),
		\]
		where \(N\) is the outward normal to \(BC\), up to normalization, and \(T\) is a
		tangent vector to \(BC\). Differentiating the Neumann condition
		\[
		u_x+q u_y=0
		\]
		along \(T\), we get
		\[
		(q\partial_x-\partial_y)(u_x+q u_y)=0.
		\]
		Thus
		\[
		q u_{xx}+(q^2-1)u_{xy}-q u_{yy}=0.
		\]
		Using \(u_{xx}+u_{yy}=-\mu u\), this becomes
		\[
		(q^2-1)u_{xy}-2q u_{yy}=q\mu u.
		\]
		In terms of \(w=u_y\), this is
		\[
		(q^2-1)w_x-2q w_y=q\mu u
		\qquad\text{on }BC.
		\]
		Equivalently,
		\[
		(-N+qT)\cdot\nabla w=q\mu u
		\qquad\text{on }BC.
		\]
		Hence
		\[
		N\cdot\nabla w=qT\cdot\nabla w-q\mu u
		\qquad\text{on }BC.
		\]
		
		We now prove that \(w\le 0\). Suppose, to the contrary, that \(w\) has a positive
		nodal component \(D\subset\{w>0\}\). Let
		\[
		\Gamma=\partial D\cap BC.
		\]
		Multiplying \(-\Delta w=\mu w\) by \(w\) and integrating over \(D\), we obtain
		\[
		\int_D |\nabla w|^2
		=
		\mu\int_D w^2
		+
		\int_{\partial D}w\,\partial_\nu w .
		\]
		On the nodal part of \(\partial D\), \(w=0\). On \(AB\cup AC\), again \(w=0\).
		Therefore the only possible boundary contribution comes from \(\Gamma\). Since the
		unit outward normal on \(BC\) is a positive multiple of \(N\), it is enough to compute
		the sign of
		\[
		\int_\Gamma w\,N\cdot\nabla w.
		\]
		Using the boundary identity above,
		\[
		\int_\Gamma w\,N\cdot\nabla w
		=
		q\int_\Gamma w\,T\cdot\nabla w
		-
		q\mu\int_\Gamma uw .
		\]
		On each connected component of \(\Gamma\), the first integral is an endpoint term:
		\[
		\int_\Gamma w\,T\cdot\nabla w
		=
		\frac12\int_\Gamma T\cdot\nabla(w^2)
		=
		0,
		\]
		because \(w=0\) at the endpoints of the nodal arc. Hence
		\[
		\int_\Gamma w\,N\cdot\nabla w
		=
		-q\mu\int_\Gamma uw
		\le 0,
		\]
		with strict inequality if \(\Gamma\) has positive length.
		
		Consequently,
		\[
		\int_D |\nabla w|^2
		\le
		\mu\int_D w^2.
		\]
		Extending \(w|_D\) by zero to all of \(\Omega_q\), we obtain an admissible test
		function for the Rayleigh quotient of \(\mu(q)\). If the inequality is strict, this
		contradicts the variational characterization of \(\mu(q)\). If equality holds, then
		the zero extension is itself a first eigenfunction. This is impossible, since the first
		eigenfunction is strictly positive in \(\Omega_q\), whereas the zero extension vanishes
		outside \(D\). Thus \(w\) has no positive nodal component, and so
		\[
		u_y=w\le 0.
		\]
		Since \(w\not\equiv0\), the strong maximum principle applied to \(-w\) gives
		\[
		u_y<0
		\qquad\text{in }\Omega_q .
		\]
		
		We next prove the monotonicity in the \(x\)-direction. Let
		\[
		z=u_x .
		\]
		Then
		\[
		-\Delta z=\mu z
		\qquad\text{in }\Omega_q .
		\]
		On the Dirichlet side \(AC\), the Hopf lemma applied to the positive eigenfunction
		\(u\) gives
		\[
		u_x>0
		\qquad\text{on }AC^\circ.
		\]
		On \(BC\), the Neumann condition gives
		\[
		z=u_x=-q u_y\ge 0,
		\]
		and in fact \(z>0\) on \(BC^\circ\). On \(AB\), differentiating \(u_y=0\) in the
		\(x\)-direction gives
		\[
		z_y=u_{xy}=0,
		\]
		or equivalently \(\partial_\nu z=0\) on \(AB\).
		
		Suppose that \(z\) has a negative nodal component
		\(D\subset\{z<0\}\). Since \(z\ge0\) on \(AC\cup BC\), the component
		\(D\) has no open boundary portion on \(AC\cup BC\). Hence the boundary
		of \(D\) consists, up to endpoints, of nodal arcs where \(z=0\) and
		possibly portions on \(AB\), where \(\partial_\nu z=0\). Therefore,
		multiplying \(-\Delta z=\mu z\) by \(z\) over \(D\) and integrating by
		parts gives
		\[
		\int_D |\nabla z|^2=\mu\int_D z^2 .
		\]
		Let
		\[
		\phi =
		\begin{cases}
			z, & \text{in } D,\\
			0, & \text{in } \Omega_q\setminus D .
		\end{cases}
		\]
		Then \(\phi\in H^1(\Omega_q)\), \(\phi=0\) on \(AC\), and
		\[
		\frac{\int_{\Omega_q}|\nabla\phi|^2}
		{\int_{\Omega_q}\phi^2}
		=\mu(q).
		\]
		Thus \(\phi\) is a minimizer for the first Rayleigh quotient. Since the
		first mixed eigenvalue is simple, \(\phi\) must be a constant multiple of
		the positive first eigenfunction. This is impossible because \(\phi\)
		vanishes identically on the nonempty open set \(\Omega_q\setminus D\).
		Hence no negative nodal component exists, and \(z\ge0\) in \(\Omega_q\). Since \(z\not\equiv0\), the strong maximum principle gives
		\[
		u_x=z>0
		\qquad\text{in }\Omega_q .
		\]
		The theorem follows.
	\end{proof}
	
	We choose the positive first eigenfunction \(u=u_q\) of the mixed problem with the normalization
	\[
	\int_{\Omega_q}u^2=1.
	\]
	Set
	\[
	A(q)=\int_{\Omega_q}u_x^2,\qquad
	B(q)=\int_{\Omega_q}u_y^2.
	\]
	Then
	\[
	\mu(q)=A(q)+B(q).
	\]
	
	We shall use the previous monotonicity result of $u_q$ to prove the monotonicity of the following normalized first eigenvalue with respect to $q$.

	\begin{proof}[Proof of Theorem \ref{monoeigenvalue}]
		We use the same deformation as in the torsion case:
		\[
		\Phi_s(x,y)=\left(\frac{q+s}{q}x,y\right).
		\]
		This maps \(\Omega_q\) onto \(\Omega_{q+s}\). Its velocity field at \(s=0\) is
		\[
		\eta(x,y)=\left(\frac{x}{q},0\right).
		\]
		The Dirichlet side \(AC\) is fixed, and on the horizontal Neumann side \(AB\) one has
		\[
		\eta\cdot\nu=0.
		\]
		Thus only the slanted Neumann side \(BC\) contributes to the shape derivative.
		
		For a simple mixed eigenvalue with normalized eigenfunction \(\int_{\Omega_q}u^2=1\),
		Hadamard's formula gives
		\[
		\mu'(q)
		=
		\int_{\partial\Omega_q} G(u)\,\eta\cdot\nu\,ds,
		\]
		where
		\[
		G(u)=-(\partial_\nu u)^2
		\quad\text{on the Dirichlet part},
		\]
		and
		\[
		G(u)=|\nabla_\tau u|^2-\mu(q)u^2
		\quad\text{on the Neumann part}.
		\]
		Since \(AC\) is fixed and \(\eta\cdot\nu=0\) on \(AB\), only \(BC\) contributes. On
		\(BC\), \(\partial_\nu u=0\), so
		\[
		|\nabla_\tau u|^2=|\nabla u|^2.
		\]
		Parametrize \(BC\) by
		\[
		(x,1-x/q),\qquad 0<x<q.
		\]
		Then
		\[
		\nu=\frac{(1,q)}{\sqrt{1+q^2}},
		\qquad
		ds=\frac{\sqrt{1+q^2}}{q}\,dx,
		\]
		and therefore
		\[
		\eta\cdot\nu\,ds
		=
		\frac{x}{q^2}\,dx.
		\]
		Hence
		\[
		\mu'(q)
		=
		\frac1{q^2}
		\int_0^q
		x\left(|\nabla u|^2-\mu(q)u^2\right)
		\left(x,1-\frac{x}{q}\right)\,dx.
		\]
		We now convert this boundary integral into a bulk integral.
		
		Use \(x u_x\) as a test function in the weak formulation. Since \(x=0\) on the
		Dirichlet side, \(x u_x\) is admissible. Thus
		\[
		\int_{\Omega_q}\nabla u\cdot\nabla(xu_x)
		=
		\mu(q)\int_{\Omega_q}xuu_x.
		\]
		The left-hand side is
		\[
		\begin{aligned}
			\int_{\Omega_q}\nabla u\cdot\nabla(xu_x)
			&=
			\int_{\Omega_q}
			\left(
			u_x^2+xu_xu_{xx}+xu_yu_{xy}
			\right)\\
			&=
			A(q)+\frac12\int_{\Omega_q}x\partial_x(|\nabla u|^2).
		\end{aligned}
		\]
		By the divergence theorem,
		\[
		\int_{\Omega_q}x\partial_x(|\nabla u|^2)
		=
		\int_{\partial\Omega_q}x|\nabla u|^2\nu_x\,ds
		-
		\int_{\Omega_q}|\nabla u|^2.
		\]
		Only the slanted side contributes to the boundary term, and on \(BC\),
		\[
		\nu_x\,ds=\frac1q\,dx.
		\]
		Therefore
		\[
		\int_{\Omega_q}\nabla u\cdot\nabla(xu_x)
		=
		\frac{A(q)-B(q)}2
		+
		\frac1{2q}
		\int_0^q
		x|\nabla u|^2
		\left(x,1-\frac{x}{q}\right)\,dx.
		\]
		
		On the other hand,
		\[
		\int_{\Omega_q}xuu_x
		=
		\frac12\int_{\Omega_q}x\partial_x(u^2).
		\]
		Again by the divergence theorem,
		\[
		\int_{\Omega_q}xuu_x
		=
		\frac1{2q}
		\int_0^q
		xu^2
		\left(x,1-\frac{x}{q}\right)\,dx
		-
		\frac12\int_{\Omega_q}u^2.
		\]
		Using the normalization \(\int_{\Omega_q}u^2=1\), this becomes
		\[
		\int_{\Omega_q}xuu_x
		=
		\frac1{2q}
		\int_0^q
		xu^2
		\left(x,1-\frac{x}{q}\right)\,dx
		-
		\frac12.
		\]
		
		Combining the two identities gives
		\[
		\frac{A(q)-B(q)}2
		+
		\frac1{2q}
		\int_0^q
		x|\nabla u|^2
		\left(x,1-\frac{x}{q}\right)\,dx
		=
		\frac{\mu(q)}{2q}
		\int_0^q
		xu^2
		\left(x,1-\frac{x}{q}\right)\,dx
		-
		\frac{\mu(q)}2.
		\]
		Since
		\[
		\mu(q)=A(q)+B(q),
		\]
		we obtain
		\[
		\frac1q
		\int_0^q
		x\left(|\nabla u|^2-\mu(q)u^2\right)
		\left(x,1-\frac{x}{q}\right)\,dx
		=
		-2A(q).
		\]
		Substituting this into the Hadamard formula yields
		\[
		\mu'(q)
		=
		-\frac{2A(q)}{q}.
		\]
		
		Therefore
		\[
		\frac{d}{dq}\bigl(q\mu(q)\bigr)
		=
		\mu(q)+q\mu'(q)
		=
		A(q)+B(q)-2A(q)
		=
		B(q)-A(q).
		\]
		Equivalently,
		\[
		\frac{d}{dq}\bigl(q\mu(q)\bigr)
		=
		-\bigl(A(q)-B(q)\bigr).
		\]
		
		It remains to prove
		\[
		A(q)-B(q)>0.
		\]
		Use \(yu_y\) as a test function. Since \(u=0\) on \(AC\), its tangential derivative
		there is zero, so \(u_y=0\) on \(AC\) in the trace sense. Hence \(yu_y\) is admissible.
		The weak formulation gives
		\[
		\int_{\Omega_q}\nabla u\cdot\nabla(yu_y)
		=
		\mu(q)\int_{\Omega_q}yuu_y.
		\]
		Now
		\[
		\begin{aligned}
			\int_{\Omega_q}\nabla u\cdot\nabla(yu_y)
			&=
			\int_{\Omega_q}
			\left(
			yu_xu_{xy}+u_y^2+yu_yu_{yy}
			\right)\\
			&=
			B(q)+\frac12\int_{\Omega_q}y\partial_y(|\nabla u|^2).
		\end{aligned}
		\]
		By the divergence theorem,
		\[
		\int_{\Omega_q}y\partial_y(|\nabla u|^2)
		=
		\int_{\partial\Omega_q}y|\nabla u|^2\nu_y\,ds
		-
		\int_{\Omega_q}|\nabla u|^2.
		\]
		Hence
		\[
		\int_{\Omega_q}\nabla u\cdot\nabla(yu_y)
		=
		\frac{B(q)-A(q)}2
		+
		\frac12
		\int_{\partial\Omega_q}y|\nabla u|^2\nu_y\,ds.
		\]
		Thus
		\[
		A(q)-B(q)
		=
		\int_{\partial\Omega_q}y|\nabla u|^2\nu_y\,ds
		-
		2\mu(q)\int_{\Omega_q}yuu_y.
		\]
		The boundary term is nonnegative. Indeed, on \(AB\) one has \(y=0\), on \(AC\) one has
		\(\nu_y=0\), and on the slanted side \(BC\) one has \(y>0\) and \(\nu_y>0\). Therefore
		\[
		\int_{\partial\Omega_q}y|\nabla u|^2\nu_y\,ds\ge0.
		\]
		On the other hand, by the directional monotonicity,
		\[
		u_y<0\qquad\text{in }\Omega_q.
		\]
		Since \(u>0\) and \(y>0\) in \(\Omega_q\), by Theorem \ref{monoeigenfunction}, we have
		\[
		\int_{\Omega_q}yuu_y<0.
		\]
		Consequently,
		\[
		A(q)-B(q)>0.
		\]
		Therefore
		\[
		\frac{d}{dq}\bigl(q\mu(q)\bigr)
		=
		-\bigl(A(q)-B(q)\bigr)
		<0.
		\]
		This proves the strict monotonicity of \(q\mu(q)\).
	\end{proof}
	
	\section{Monotonicity for mixed Dirichlet--Robin torsion}
	\label{mixedrobinsection}
		\begin{theorem}[Upper-half monotonicity for mixed Dirichlet--Robin torsion]
		Let \(\beta\ge 0\) be a constant. Let \(D\subset\mathbb R^2\) be a bounded
		Lipschitz domain which is symmetric with respect to the \(x\)-axis and vertically
		convex. Let
		\begin{equation*}
			\partial D=\overline{\Gamma_D}\cup\overline{\Gamma_R},
			\qquad
			\Gamma_D\cap\Gamma_R=\emptyset,
			\qquad
			\Gamma_D\neq\emptyset,
		\end{equation*}
		where \(\Gamma_D\) and \(\Gamma_R\) are relatively open in \(\partial D\) and are
		symmetric with respect to the \(x\)-axis. Let \(u\) solve
		\begin{equation*}
			\begin{cases}
				-\Delta u=1 & \text{in } D,\\
				u=0 & \text{on } \Gamma_D,\\
				\partial_\nu u+\beta u=0 & \text{on } \Gamma_R.
			\end{cases}
		\end{equation*}
		Put
		\begin{equation*}
			D^+=D\cap\{y>0\},
			\qquad
			\Gamma_D^+=\Gamma_D\cap\{y>0\},
			\qquad
			\Gamma_R^+=\Gamma_R\cap\{y>0\}.
		\end{equation*}
		
		Assume that every smooth arc of \(\partial D\cap\{y>0\}\) is of class
		\(C^{2,\alpha}\). Assume also that the only nonsmooth boundary points of
		\(\partial D\cap\{y>0\}\) are Dirichlet--Robin junctions. At each such junction,
		assume that the tangent cone is non-reentrant, that is, its opening angle is at
		most \(\pi\).
		
		On the smooth part of \(\Gamma_R^+\), define the signed curvature \(\kappa\) by
		\begin{equation*}
			D_\tau\nu=\kappa\tau,
		\end{equation*}
		where \(\nu\) is the exterior unit normal and \(\tau\) is a unit tangent. Assume
		that
		\begin{equation*}
			\kappa+\beta\ge 0
		\end{equation*}
		on the smooth part of \(\Gamma_R^+\). Then
		\begin{equation*}
			u_y\le 0 \quad \text{in } D^+.
		\end{equation*}
		Moreover, if \(u_y\not\equiv 0\) in \(D^+\), then
		\begin{equation*}
			u_y<0 \quad \text{in } D^+.
		\end{equation*}
	\end{theorem}
	
	\begin{proof}
		By the maximum principle for mixed Dirichlet--Robin problems, using
		\(\beta\ge 0\) and \(\Gamma_D\neq\emptyset\), we have
		\begin{equation*}
			u>0 \quad \text{in } D.
		\end{equation*}
		Since \(D\), \(\Gamma_D\), and \(\Gamma_R\) are symmetric with respect to the
		\(x\)-axis, uniqueness implies that \(u\) is even in \(y\). Hence, for
		\begin{equation*}
			w:=u_y,
		\end{equation*}
		we have
		\begin{equation*}
			w=0 \quad \text{on } D\cap\{y=0\}.
		\end{equation*}
		Differentiating \(-\Delta u=1\) with respect to \(y\) gives
		\begin{equation*}
			\Delta w=0 \quad \text{in } D^+.
		\end{equation*}
		
		We first treat the Dirichlet part of the upper boundary. On every smooth arc of
		\(\Gamma_D^+\), the tangential derivative of \(u\) vanishes. Hence
		\begin{equation*}
			w=u_y=(e_y\cdot\nu)u_\nu+(e_y\cdot\tau)u_\tau=\nu_y u_\nu,
		\end{equation*}
		where \(e_y=(0,1)\) and \(\nu_y=e_y\cdot\nu\). By vertical convexity,
		\begin{equation*}
			\nu_y\ge 0
		\end{equation*}
		on the upper boundary. By Hopf's lemma,
		\begin{equation*}
			u_\nu<0
		\end{equation*}
		on the relatively open Dirichlet arcs. Therefore
		\begin{equation*}
			w\le 0 \quad \text{on } \Gamma_D^+.
		\end{equation*}
		
		Next we exclude positive maxima at Dirichlet--Robin junctions in the open upper
		half-plane. Let \(P\in\overline{\Gamma_D^+}\cap\overline{\Gamma_R^+}\) be such
		a junction, and let \(\omega\le\pi\) be the opening angle of the tangent cone of
		\(D\) at \(P\).
		
		If \(\omega\le\pi/2\), standard regularity for mixed boundary value problems in
		polygonal corners gives continuity of \(\nabla u\), and hence of \(w\), up to
		\(P\). Since \(w\le 0\) on the adjacent Dirichlet arc, a positive maximum of
		\(w\) cannot occur at \(P\).
		
		It remains to consider the case
		\begin{equation*}
			\pi/2<\omega\le\pi.
		\end{equation*}
		After translating \(P\) to the origin and rotating the coordinates, the tangent
		cone may be written as
		\begin{equation*}
			K_\omega=\{(r,\theta):r>0,\ 0<\theta<\omega\},
		\end{equation*}
		with the Dirichlet side corresponding to \(\theta=0\) and the Robin side
		corresponding to \(\theta=\omega\). The Robin term \(\beta u\) is lower order in
		the corner expansion. Thus the leading mixed Dirichlet--Robin singular mode is
		the mixed Dirichlet--Neumann mode
		\begin{equation*}
			r^\lambda\sin(\lambda\theta),
			\qquad
			\lambda=\frac{\pi}{2\omega}\in\left[\frac12,1\right).
		\end{equation*}
		Consequently, near \(P\),
		\begin{equation*}
			u(r,\theta)=a r^\lambda\sin(\lambda\theta)+R(r,\theta),
			\qquad
			|\nabla R(r,\theta)|=o(r^{\lambda-1}).
		\end{equation*}
		Since \(u>0\) in \(D\), either \(a>0\), or else this leading singular coefficient
		vanishes.
		
		If \(a=0\), then the first singular term with exponent below \(1\) is absent.
		The next mixed corner exponent is
		\begin{equation*}
			\frac{3\pi}{2\omega}>1.
		\end{equation*}
		Thus \(\nabla u\), and hence \(w\), has a finite limit at \(P\). Since
		\(w\le 0\) on the adjacent Dirichlet arc, a positive maximum of \(w\) cannot occur
		at \(P\).
		
		Assume now that \(a>0\). Let \(\phi\) be the angle of the vertical direction
		\(e_y\) in the above polar coordinates. We choose the branch so that \(e_y\)
		lies on the exterior side of the upper tangent cone adjacent to the Dirichlet
		side. By vertical convexity of \(D\),
		\begin{equation*}
			\omega-\frac{3\pi}{2}\le \phi\le 0.
		\end{equation*}
		For \(0<\theta<\omega\), since \(\lambda\omega=\pi/2\), we have
		\begin{equation*}
			-\pi<\phi+(\lambda-1)\theta<0.
		\end{equation*}
		Indeed, the upper bound follows from \(\phi\le 0\) and \(\lambda<1\), while the
		lower bound follows from
		\begin{equation*}
			\phi+(\lambda-1)\theta
			>
			\phi+(\lambda-1)\omega
			=
			\phi+\frac{\pi}{2}-\omega
			\ge 
			-\pi.
		\end{equation*}
		Differentiating the leading singular term in the \(e_y\)-direction gives
		\begin{equation*}
			\partial_y\left(r^\lambda\sin(\lambda\theta)\right)
			=
			\lambda r^{\lambda-1}
			\sin\bigl(\phi+(\lambda-1)\theta\bigr).
		\end{equation*}
		The sine factor is strictly negative for \(0<\theta<\omega\). Since \(a>0\) and
		\begin{equation*}
			|\nabla R(r,\theta)|=o(r^{\lambda-1}),
		\end{equation*}
		we obtain
		\begin{equation*}
			u_y<0
		\end{equation*}
		in a punctured neighborhood of \(P\) inside \(D^+\). Hence a positive maximum of
		\(w\) cannot occur at \(P\).
		
		We now run the maximum-principle argument. Suppose, to the contrary, that
		\begin{equation*}
			M:=\max_{\overline{D^+}} w>0.
		\end{equation*}
		The strong maximum principle excludes an interior maximum. The symmetry axis
		satisfies \(w=0\), the Dirichlet part satisfies \(w\le 0\), and the preceding
		corner analysis excludes positive maxima at all nonsmooth boundary points of
		\(\partial D\cap\{y>0\}\). By assumption, there are no other nonsmooth points on
		the upper boundary. Therefore \(M\) is attained at a smooth point
		\begin{equation*}
			P\in\Gamma_R^+.
		\end{equation*}
		
		At \(P\), write
		\begin{equation*}
			e_y=\alpha\tau+\gamma\nu,
			\qquad
			\gamma=e_y\cdot\nu=\nu_y\ge 0.
		\end{equation*}
		The Robin condition gives
		\begin{equation*}
			u_\nu=-\beta u.
		\end{equation*}
		Differentiating the Robin condition tangentially and using \(D_\tau\nu=\kappa\tau\),
		we obtain
		\begin{equation*}
			u_{\tau\nu}+(\kappa+\beta)u_\tau=0
			\quad \text{on } \Gamma_R^+,
		\end{equation*}
		where \(u_{\xi\eta}=D^2u[\xi,\eta]\).
		
		Since \(P\) is a boundary maximum point of the harmonic function \(w\), Hopf's
		lemma gives
		\begin{equation*}
			\partial_\nu w(P)>0,
		\end{equation*}
		while the tangential derivative satisfies
		\begin{equation*}
			D_\tau w(P)=0.
		\end{equation*}
		On the other hand,
		\begin{equation*}
			M=w(P)=\alpha u_\tau+\gamma u_\nu
			=\alpha u_\tau-\gamma\beta u.
		\end{equation*}
		Thus
		\begin{equation*}
			\alpha u_\tau=M+\gamma\beta u>0.
		\end{equation*}
		In particular, \(\alpha\neq 0\).
		
		From \(D_\tau w(P)=0\), we have
		\begin{equation*}
			0=D_\tau w
			=D^2u[\tau,e_y]
			=\alpha u_{\tau\tau}+\gamma u_{\tau\nu}.
		\end{equation*}
		Using
		\begin{equation*}
			u_{\tau\nu}=-(\kappa+\beta)u_\tau,
		\end{equation*}
		we get
		\begin{equation*}
			\alpha^2 u_{\tau\tau}
			=
			\gamma(\kappa+\beta)\alpha u_\tau
			=
			\gamma(\kappa+\beta)(M+\gamma\beta u)
			\ge 0.
		\end{equation*}
		Hence
		\begin{equation*}
			u_{\tau\tau}\ge 0.
		\end{equation*}
		Since
		\begin{equation*}
			u_{\tau\tau}+u_{\nu\nu}=\Delta u=-1,
		\end{equation*}
		we have
		\begin{equation*}
			u_{\nu\nu}<0.
		\end{equation*}
		Therefore
		\begin{equation*}
			\begin{aligned}
				\partial_\nu w
				&=D^2u[\nu,e_y] \\
				&=\alpha u_{\tau\nu}+\gamma u_{\nu\nu} \\
				&=-(\kappa+\beta)\alpha u_\tau+\gamma u_{\nu\nu} \\
				&=-(\kappa+\beta)(M+\gamma\beta u)+\gamma u_{\nu\nu} \\
				&\le 0.
			\end{aligned}
		\end{equation*}
		This contradicts \(\partial_\nu w(P)>0\). Hence \(M>0\) is impossible, and therefore
		\begin{equation*}
			u_y=w\le 0 \quad \text{in } D^+.
		\end{equation*}
		
		Finally, if \(u_y\not\equiv 0\) in \(D^+\) and \(u_y\) vanished at an interior point
		of \(D^+\), then the harmonic function \(w=u_y\le 0\) would attain its maximum
		\(0\) in the interior. The strong maximum principle would imply \(w\equiv 0\), a
		contradiction. Thus
		\begin{equation*}
			u_y<0 \quad \text{in } D^+.
		\end{equation*}
	\end{proof}

	\section*{AI Assistance Statement}
	
	The authors used AI tools only for auxiliary symbolic manipulation and exploratory checking of formal series expansions. All arguments, identities, estimates, and final text were independently verified and are the sole responsibility of the authors.

\end{document}